\documentclass[11pt]{article}
\usepackage[a4paper,margin=3cm]{geometry}
\usepackage[utf8]{inputenc} % Encodage des caractères du fichier-source.
\usepackage[T1]{fontenc} % Encodage des caractères en sortie.
\usepackage{lmodern}
\usepackage[autolanguage]{numprint} %% Formatage des nombres.
\usepackage{hyperref} % Génère des liens hypertexte dans le fichier pdf.
\usepackage{graphicx} % Permet l'insertion d'images.
\usepackage[all]{xy}
\usepackage{amsmath,amssymb,amsthm}

\newcommand\NN{\mathbb{N}} % Ensemble des entiers naturels.
\newcommand\RR{\mathbb{R}} % Ensemble des nombres réels.
\newcommand\CC{\mathbb{C}} % Ensemble des nombres complexes.
\newcommand\PP{\mathbb{P}} % proba
\newtheorem{theorem}{Theorem}[section]%
\newtheorem{proposition}[theorem]{Proposition}%
\newtheorem{lemma}[theorem]{Lemma}%

\newtheorem{definition}[theorem]{Definition}%
\newtheorem{remark}[theorem]{Remark}%

\author{Raphaël Butez\footnote{CEREMADE, Université Paris Dauphine, butez@ceremade.dauphine.fr}}

\title{Large deviations for the empirical measure of random polynomials: revisit of the Zeitouni-Zelditch theorem.}

\begin{document} % Début du document proprement dit.

\maketitle % Crée le bandeau de titre

\begin{abstract}
	This article revisits the work by Ofer Zeitouni and Steve Zelditch on large deviations for the empirical measures of random orthogonal polynomials with i.i.d. Gaussian complex coefficients, and extends this result to real Gaussian coefficients. This article does not require any knowledge in geometry. For clarity, we focus on two classical cases: Kac polynomials and elliptic polynomials.

\end{abstract}
\tableofcontents

\section{Introduction}
We study three different models of random polynomials, orthogonal polynomials, Kac polynomials, and elliptic polynomial, the two last being examples of orthogonal polynomials. The coefficients are i.i.d. random variables which can be either:

\begin{tabular}{l@{\hspace{2cm}}l}
Complex Gaussian coefficients & $a_k= b_k+i c_k$ where $\left(\begin{matrix} b_k \\ c_k \end{matrix}\right) \sim \mathcal{N}(0,\frac{1}{2}I_2)$ \\
	Real Gaussian coefficients & $a_k \sim \mathcal{N}(0,\frac{1}{2}). $\\
\end{tabular} 
\newline
\newline
We will always refer to those two possibilities as the complex and real case. In all the article, we assume that the $a_k$'s are independent. Given $(R_0,\dots,R_n)$ a basis of $\CC_n[X]$ we consider the random polynomials:
\begin{equation}
P_n= a_0R_0 + a_1R_1+ \dots+a_n R_n.
\end{equation}
In order to study the zeros $z_1, \dots ,z_n$ of the random polynomials $P_n$ we introduce their empirical measure:
\begin{equation}
\mu_n = \frac{1}{n} \sum_{i=1}^{n} \delta_{z_i}.  
\end{equation} 
We will focus on three classes of random polynomials corresponding to different choices of the polynomials $R_j$'s:
\begin{figure}[h!]
\centering
\begin{tabular}{l@{\hspace{1cm}}l}
Orthogonal polynomials & $R_k$  orthonormal family in $L^2$ \vspace{0.06cm} \\ 
Kac polynomials & $R_k= X^k $  \\
Elliptic polynomials & $R_k= \sqrt{\binom{n}{k}}X^k$ \\

\end{tabular}
\end{figure}

The study of the zeros of random polynomials started with articles by Kac \cite{kac}, Littlewood and Offord \cite{littlehoodofford}, Hammersley \cite{hammersley} which focused on the number of real zeros. The literature about random polynomials is vast, we refer to the book by Bharucha-Reid and Sambandham \cite{bharucha} and the article by Tao and Vu \cite{taovu} for a nice account of the classical results. The study of the complex roots was initiated by Polya, Sparo and Sur. Recently, the minimal condition to obtain the convergence of $(\mu_n)_{n \in \NN}$ was given by Kabluchko and Zaporozhets \cite{kabluchkozapopoly}.

The purpose of this article is to revisit the article of Zeitouni and Zelditch \cite{zeitounizelditch}. They prove that the empirical measures of the zeros of random orthonormal polynomials with respect to a  scalar product in $L^2$ and complex Gaussian coefficients satisfy a large deviation principle in the projective space $\mathbb{CP}^1$. Here we revisit their proof of their theorem in an elementary way although the techniques used are mainly a reformulation of their work. Using a compactification technique based on inverse stereographic projection, we prove a large deviation principle for the push-forward problem on a sphere of $\RR^3$ and then obtain the result in $\CC$. The proof adapts to the case of real coefficients, which allows us to extend the theorem. The compactification technique was first introduced by Zelditch in \cite{zeitounizelditch}, and discovered again independently by Hardy \cite{hardy} in order to prove a large deviation principle for Coulomb gases with weakly confining potential. The compactification method was also used by Bloom in \cite{bloom} in a more general framework.

Large deviations for empirical measures of random polynomials are only known for Gaussian complex coefficients \cite{zeitounizelditch}, which is the subject of the present work, and for exponential coefficients in the Kac case studied by Ghosh and Zeitouni in \cite{ghoshzeitouni}. All these cases rely on the ability to compute the law of the roots of $P_n$. These results should be compared with their equivalent in random matrix theory: the Ginibre ensemble, real or complex. Many authors used the link between Coulomb gases and eigenvalues of random matrices to obtain large deviation principles as in Ben Arous and Guionnet \cite{benarousguionnet}, Ben Arous and Zeitouni \cite{benarouszeitouni}, Hiai and Petz \cite{hiaipetz}. In a more general setup, large deviation principle for empirical measures of a Coulomb gas are valid. See for example \cite{chafai} for a similar result in any dimension with general repulsion, Hardy \cite{hardy}, or Bloom \cite{bloom}.

\subsection*{Orthogonal polynomials} 

Given a probability measure $\nu$ and a continuous function $\phi$, we consider the scalar products on $\CC_n[X]$:
\begin{equation}\label{produitscal}
\langle P,Q \rangle= \int P(z) \overline{Q(z)} e^{-n\phi(z)}d\nu(z).
\end{equation}
Let $R_0,\dots,R_n$ be an orthonormal basis for this scalar product, we define
\begin{equation}\label{polyortho}
P_n= \sum_{k=0}^{n}a_k R_k.
\end{equation}
We call $K$ the support of $\nu$. We assume that its compactification by inverse stereographic projection is non-thin at all the points of its closure. This notion comes from potential theory and is detailed in \cite[p.\ 78]{ransford1995potential}. We can understand it as the requirement that the support of $\nu$ is not too degenerated. For instance, if the support of $\nu$ is connected and has more than one point, it is non-thin at all its points \cite[Thoerem 3.8.3 p 79]{ransford1995potential}. It also holds if it has a finite number of connected components with more than one point. On the other hand, a polar set is thin at every point.
We define the Berstein-Markov property which was introduced in \cite{zeitounizelditch}. This property is the key of the proof of the large deviations upper bound.
\begin{definition}[Bernstein-Markov property]\label{bersteinmarkov}
	We say that the couple $(\phi,\mu)$ satisfies the Bernstein-Markov property if, for every $\varepsilon>0$, there exists a constant $C_{\varepsilon}>0$ such that, for any $n\in \NN$ and for any polynomial $P \in \CC_n[X]$ we have:
	\begin{equation*}
	\sup_{z\in K} \{ |P(z)|^2e^{-n\phi(z)}\} \leq C_{\varepsilon}e^{\varepsilon n} \|P\|_{L^2}^2
	\end{equation*}
where $K$ is the support of the measure $\nu$.
\end{definition}
We define the Hamiltonian:
\begin{equation} \label{hamiltonianortho}
H_O(z_1,\dots,z_n)=-\frac{1}{n^2}\sum_{i\neq j} \log|z_i-z_j| + \frac{n+1}{n^2}\log \int \prod_{i=1}^{n}|z-z_i|^2 e^{-n\phi(z)}d\nu(z).
\end{equation}
In the complex case, the distribution of the roots $(z_1,\dots,z_n)$ of $P_n$ is given by:
\begin{equation}\label{gazgeneral}
\frac{1}{Z_n}\exp \left(-\beta_n H_O(z_1,\dots,z_n)\right)d\ell_{\CC^n}(z_1,\dots,z_n).
\end{equation}
where $\ell_{\CC^n}$ is the Lebesgue measure on $\CC^n$, and $Z_n$ is a constant. 
$\beta_n$ is the inverse of a temperature, so we can see $1/\beta_n$ as a cooling scheme. 
In this article, we will always consider: 
 $$\beta_n = n^2$$ 
 which corresponds to the distribution of the roots of random polynomials. We can see $(z_1,\dots,z_n)$ as a system of particles in interaction. The term $$-\frac{1}{n^2}\sum_{i\neq j} \log|z_i-z_j|$$ corresponds to a repulsion between the particles and is compensated by the confinement $$\frac{n+1}{n^2}\log \int \prod_{i=1}^{n}|z-z_i|^2 e^{-n\phi(z)}d\nu(z).$$ This model is very close to the classical Coulomb gas model, where the confinement takes the simpler form $\sum_{i=1}^{n} V(z_i)$, which does not involve any interaction between the particles. This non-interaction property can be seen as linearity with respect to the empirical measure via the relation:
 \[ \frac{1}{n} \sum_{i=1}^{n} V(z_i)= \int V(w) d\mu_n(w)= \langle V , \mu_n\rangle.\]
The confinement term associated to the Hamiltonian \eqref{hamiltonianortho} is more complicated, but can still be compared to a classical potential thanks to the Jensen inequality.

The study of the real case is interesting only if the polynomials $R_k$'s are real.
In the real case, the distribution of the roots is not absolutely continuous with respect to the Lebesgue measure of $\CC^n$ as the probability to have a real root is positive. This distribution is given by the following mixture:
\begin{equation}\label{gazgeneral2}
\sum_{k=0}^{\lfloor n/2 \rfloor} \frac{1}{Z_{n,k}}  \exp\left( -\beta_n\frac{1}{2}H_O(z_1,\dots,z_n)\right) \\   d\ell_{n,k}(z_1,\dots,z_n).
\end{equation}
where we defined, $\ell_{\RR}$ and $\ell_{\CC}$ being the Lebesgue measures on $\RR$ and $\CC$,
\begin{equation}
d\ell_{n,k}(z_1,\dots,z_n)= d\ell_{\RR}(z_1) \dots   d\ell_{\RR}(z_{n-2k}) d\ell_{\CC}(z_{n-k})  \dots  d\ell_{\CC}(z_n)
\end{equation}
and where $Z_{n,k}$ are constants. The first $n-2k$ particles are on the real line and with $k$ pairs of complex numbers and their conjugates. In the complex case, all the results of this article are valid for any sequence satisfying
$$\beta_n \gg n.$$ In the real case, additional assumptions are needed, they are given in \eqref{assumption0}, \eqref{assumption} and \eqref{asssumption2}. Those asumptions correspond to a uniform control of the normalizing constants

In this article, the term weak topology corresponds to the topology of convergence in distribution, which is the weak topology associated to continuous and bounded test functions. This topology is associated to the Bounded Lipschitz metric $d$ defined as:
\[ \forall \mu , \nu \quad d(\mu,\nu)= \sup_{f} \left| \int f \mu - \int f d\nu \right| \]
where the surpremum is taken over functions bounded by $1$ and $1$-Lipschitz.

\begin{theorem}[Large deviation principle for complex orthogonal polynomials]\label{LDPgeneral}
	Let $\mu_n$ be the empirical measure of the gas \eqref{gazgeneral}.
	Let	us define $I_O: \mathcal{M}_1(\CC) \rightarrow \RR\cup \{\infty\}$:
	\begin{multline*}
	I_O(\mu)=  -\iint \left(\log|z-w|- \frac{1}{2}\log(1+|z|^2) -\frac{1}{2}\log(1+|w|^2)\right)d\mu(z)d\mu(w) \\ 
	+ \sup_{z \in K} \left[ \int \log|z-w|^2-\log(1+|w|^2) d\mu(w)- \phi(z) \right].
	\end{multline*}
	When $\displaystyle \int\log(1+|w|^2) d\mu(w)<\infty$ then we have:
	$$ I_O(\mu) =  -\iint \log|z-w|d\mu(z)d\mu(w)  + \sup_{z \in K} \left[ \int \log|z-w|^2 d\mu(w)- \phi(z) \right]$$
	If the couple $(\phi,\nu)$ satisfies the Bernstein-Markov property \eqref{bersteinmarkov} then $(\mu_n)_{n\in \NN}$ satisfy a large deviation principle in $\mathcal{M}_1(\CC)$ with the weak topology, speed $\beta_n$ and good rate function $I_O-\inf I_O$ . This means that for any Borel set $A \subset \mathcal{M}_1(\CC)$ we have:
	\begin{equation*}
	-\inf_{\mathrm{Int}A} (I_O-\inf I_O) \leq \varliminf_{n\to\infty} \frac{1}{\beta_n} \log \PP(\mu_n \in A) \leq \varlimsup_{n\to\infty} \frac{1}{\beta_n} \log \PP(\mu_n \in A) \leq -\inf_{\mathrm{Clo}A}( I_O-\inf I_O).
	\end{equation*}
\end{theorem}

\begin{theorem}[Large deviation principle for real orthogonal polynomials]\label{LDPgeneralreel}
	Let $\mu_n$ be the empirical measure of the gas	\eqref{gazgeneral2}.
	If the couple $(\phi, \nu)$ satisfies the Bernstein-Markov property, then $(\mu_n)_{n\in \NN}$ satisfies a large deviation principle in $\mathcal{M}_1(\CC)$ with the weak topology, speed $\beta_n$ and good rate function:
	\begin{align*}
	\tilde{I}_O(\mu)= \begin{cases}
	\frac{1}{2}(I_O(\mu)-\inf I_O) &  \text{if $\mu$ is invariant under the map } z \mapsto \bar{z} \\
	\infty & \text{otherwise.}
	\end{cases} 
	\end{align*}
	This means that for any Borel set $A \subset \mathcal{M}_1(\CC)$ we have:
	\begin{equation*}
	-\inf_{\mathrm{Int}A} \tilde{I}_O \leq \varliminf_{n\to\infty} \frac{1}{\beta_n} \log \PP(\mu_n \in A) \leq \varlimsup_{n\to\infty} \frac{1}{\beta_n} \log \PP(\mu_n \in A) \leq -\inf_{\mathrm{Clo}A} \tilde{I}_O.
	\end{equation*}
\end{theorem}

Those last two theorems imply that, in both cases, almost surely:
\begin{equation}\label{convergence}
d(\mu_n,\mathrm{Argmin}(I_O))\xrightarrow[n \rightarrow \infty]{} 0
\end{equation} where $\mathrm{Argmin}(I_O)$ is the unique minizer of function $I_O$. This is a consequence of the Borel-Cantelli Lemma used with the sets $\{\mu \in \mathcal{M}_1(\CC) \mid d(\mu , \mathrm{Argmin}(I_O))>\varepsilon\}$. The minimizer is the equilibrium measure of the support of $\nu$, see \cite[Lemma 30]{zeitounizelditch}.

\subsection*{Kac polynomials} 

The most important example of orthonormal polynomials are Kac polynomials:
\begin{equation}
P_n= a_o + a_1 X +\dots + a_nX^n.
\end{equation}
The canonical basis is orthonormal with respect to the scalar product on $\CC_n[X]$:
\begin{equation}\label{psKac}
\langle P, Q \rangle = \int P(z) \overline{Q(z)}d\nu_S(z)
\end{equation}
where $\nu_S$ is the uniform measure on $S$, the unit circle of $\CC$. The sequence $(\mu_n)_{n\in \NN}$ converges almost surely weakly towards the measure $\nu_S$. Although this result is quite ancient, we can deduced it from \eqref{convergence}. We define the Hamiltonian:
\begin{equation}\label{hamiltonienkac}
H(z_1,\dots,z_n)= -\frac{1}{n^2}\sum_{i\neq j} \log|z_i-z_j| + \frac{n+1}{n^2}\log \int \prod_{i=1}^{n}|z-z_i|^2d\nu_S(z)
\end{equation}
In the complex case, the distribution of the roots is given by the Gibbs measure:
\begin{equation}\label{kaccomplexe}
\frac{1}{Z_n}\exp \left(-\beta_n H(z_1, \dots,z_n) \right)d\ell_{\CC^n}(z_1,\dots,z_n)
\end{equation}
In the real case, the distribution of the roots is the mixture:
\begin{equation}\label{kacreel}
\sum_{k=0}^{\lfloor n/2 \rfloor} \frac{1}{Z_{n,k}}  \exp \left( -\beta_n\frac{1}{2} H(z_1,\dots,z_n) \right)d\ell_{n,k}(z_1,\dots,z_n)
\end{equation}
where the $Z_{n,k}$ are constants.

\begin{theorem}[Large deviations for complex Kac polynomials] \label{LDPKac}
	Let $\mu_n$ be the empirical measure of the gas \eqref{kaccomplexe}.
	Let	us define $I: \mathcal{M}_1(\CC) \rightarrow \RR\cup \{\infty\}$:
	\begin{align*}
	I(\mu) & =  -\iint \left( \log|z-w|-\frac{1}{2}\log(1+|z|^2)-\frac{1}{2}\log(1+|w|^2)\right)d\mu(z)d\mu(w) \\ 
	& \qquad \qquad +  \sup_{z \in S} \int \left( \log|z-w|^2-\log(1+|w|^2)\right)d\mu(w).
	\end{align*}
	When $\displaystyle \int \log(1+|z|^2)d\mu(z)$ is finite, this function can be simplified to:
	$$I(\mu)=  -\iint \log|z-w|d\mu(z)d\mu(w) + \sup_{z \in S} \int \log|z-w|^2d\mu(w).$$
	The random sequence $(\mu_n)_{n\in \NN}$ satisfies a large deviation principle in $\mathcal{M}_1(\CC)$ with the weak topology and with speed $\beta_n$ and good rate function $I$. For any Borel set $A \subset \mathcal{M}_1(\CC)$ we have:
	\begin{equation}
	-\inf_{\mathrm{Int}A} I \leq \varliminf_{n\to\infty} \frac{1}{\beta_n} \log \PP(\mu_n \in A) \leq \varlimsup_{n\to\infty} \frac{1}{\beta_n} \log \PP(\mu_n \in A) \leq -\inf_{\mathrm{Clo}A} I.
	\end{equation}
\end{theorem}

\begin{theorem}[Large deviations for real Kac polynomials]\label{LDPKacreel}
	Let $\mu_n$ be the empirical measure of the gas \eqref{kacreel}, then the random sequence $(\mu_n)_{n\in \NN}$ satisfies a large deviation principle in $\mathcal{M}_1(\CC)$ for the weak topology  with speed $\beta_n$ and good rate function $\tilde{I}$ where:
	$$
	\tilde{I}(\mu)=\begin{cases}
	\frac{1}{2}I(\mu) & \text{if $\mu$ is invariant under the map } z \mapsto \bar{z} \\
	\infty & \text{otherwise.}
	\end{cases} 
	$$
	This means that for any Borel set $A \subset \mathcal{M}_1(\CC)$ we have:
	\begin{equation*}
	-\inf_{\mathrm{Int}A} \tilde{I} \leq \varliminf_{n\to\infty} \frac{1}{\beta_n} \log \PP(\mu_n \in A) \leq \varlimsup_{n\to\infty} \frac{1}{\beta_n} \log \PP(\mu_n \in A) \leq -\inf_{\mathrm{Clo}A} \tilde{I}.
	\end{equation*}
\end{theorem}

\subsection*{Elliptic polynomials} 

We will see how the study of Kac polynomials can be adapted to prove a large deviation principle for the empirical measure associated to the roots of polynomials of the form 
$$ P_n= \sum_{k=0}^{n}  a_k \binom{n}{k}^{1/2} X^k.$$ 
The polynomials $\sqrt{n+1} \displaystyle \binom{n}{k}^{1/2} X^k$ are orthonormal for the scalar product on $\CC_n[X]$:
$$ \langle P,Q \rangle = \int P(z)\overline{Q(z)} \frac{1}{ (1+|z|^2)^{n}} \frac{d\ell_{\CC}(z)}{\pi(1+|z|^2)^2} \ .$$
As multiplying a polynomial by a constant does not change the zeros, the factor $\sqrt{n+1}$ is omitted. It is known that the random sequence $(\mu_n)_{n\in \NN}$ converges almost surely weakly towards $$\frac{d\ell_{\CC}(z)}{\pi (1+|z|^2)^2}$$ which is called the complex Cauchy measure\footnote{or Fubini-Study measure}. It can be seen as a consequence of \eqref{convergence}.  We define the Hamiltonian:
\begin{equation}
H_E(z_1,\dots,z_n)=-\frac{1}{n^2}\sum_{i\neq j} \log|z_i-z_j| + \frac{n+1}{n^2}\log \int \frac{\prod_{i=1}^{n}|z-z_i|^2}{(1+|z|^2)^n}\frac{d\ell_{\CC}(z)}{\pi (1+|z|^2)^2}
\end{equation}
and the Gibbs measure associated to the distribution of the roots in the complex case:
\begin{equation}\label{ellipticcomplexe}
\frac{1}{Z_n}\exp \left(-\beta_n H_E(z_1,\dots,z_n) \right)d\ell_{\CC^n}(z_1,\dots,z_n).
\end{equation}
In the real case, the roots form a mixture of Coulomb gases distributed with respect to:
\begin{align} \label{elliptiquereel}
\sum_{k=0}^{\lfloor n/2 \rfloor} \frac{1}{Z_{n,k}}  \exp\left( -\beta_n\frac{1}{2}H_E(z_1,\dots,z_n)\right)  d\ell_{n,k}(z_1,\dots,z_n)
\end{align}
where the $Z_{n,k}$ are constants.

\begin{theorem}[Large deviation principle for complex elliptic polynomials]\label{LDPelliptique}
	Let $\mu_n$ be the empirical measure of the gas	\eqref{ellipticcomplexe}.
	Let	us define $I_E: \mathcal{M}_1(\CC) \rightarrow \RR\cup \{\infty\}$:
	\begin{multline*}
	I_E(\mu)=  -\iint\left( \log|z-w|-\frac{1}{2}\log(1+|z|^2)- \frac{1}{2}\log(1+|w|^2)\right)d\mu(z)d\mu(w)  \\   + \sup_{z \in \CC} \left[ \int \left(\log|z-w|^2- \log(1+|w|^2)\right) d\mu(w)- \log(1+|z|^2)\right].
	\end{multline*}
	When $\displaystyle \int \log(1+|z|^2) d\mu(z) <\infty$, we can write:
	$$I_E(\mu)= -\iint \log|z-w|d\mu(z)d\mu(w) + \sup_{z \in \CC} \left[ \int \log|z-w|^2 d\mu(w)- \log(1+|z|^2)\right].$$
	$(\mu_n)_{n\in \NN}$  satisfies a large deviation principle in $\mathcal{M}_1(\CC)$ with the weak topology and with speed $\beta_n$ and good rate function $I_E-\inf I_E$. This means that for any Borel set $A \subset \mathcal{M}_1(\CC)$ we have:
	\begin{equation*}
	-\inf_{\mathrm{Int}A} (I_E-\inf I_E) \leq \varliminf_{n\to\infty} \frac{1}{\beta_n} \log \PP(\mu_n \in A) \leq \varlimsup_{n\to\infty} \frac{1}{\beta_n} \log \PP(\mu_n \in A) \leq -\inf_{\mathrm{Clo}A} (I_E-\inf I_E).
	\end{equation*}
\end{theorem}
\begin{theorem}[Large deviation principle for real elliptic polynomials] \label{LDPelliptiquereel}
	Let $\mu_n$ be the empirical measure of the gas	\eqref{elliptiquereel}.
	$(\mu_n)_{n\in \NN}$ satisfies a large deviation principle in $\mathcal{M}_1(\CC)$ with the weak topology, speed $\beta_n$ and good rate function:
	\begin{align*}
	\tilde{I}_E(\mu)= \begin{cases}
	\frac{1}{2}(I_E(\mu)-\inf I_E) & \text{if $\mu$ is invariant under the map }z \mapsto \bar{z} \\
	\infty & \text{otherwise.}
	\end{cases} 
	\end{align*}
	This means that for any Borel set $A \subset \mathcal{M}_1(\CC)$ we have:
	\begin{equation*}
	-\inf_{\mathrm{Int}A} \tilde{I}_E \leq \varliminf_{n\to\infty} \frac{1}{\beta_n} \log \PP(\mu_n \in A) \leq \varlimsup_{n\to\infty} \frac{1}{\beta_n} \log \PP(\mu_n \in A) \leq -\inf_{\mathrm{Clo}A} \tilde{I}_E.
	\end{equation*}
\end{theorem}

\subsection*{Outline of the article}

We will give a full proof of the results for Kac polynomials, and then we will show how to adapt the proof for elliptic and orthogonal polynomials. The proofs of the previous theorems are similar, and will follow these steps:
\begin{enumerate}
\renewcommand\labelenumi{$\arabic{enumi})$}
\item Compute the distribution of the roots on $\CC^n$;
\item Use of inverse stereographic projection to push-forward every object on $\mathcal{S}^2$, the sphere in $\RR^3$ centered on $(0,0,\frac{1}{2})$ and of radius $1/2$;
\item Prove a large deviation principle in $\mathcal{M}_1(\mathcal{S}^2)$;

\item Use of contraction principle to obtain the large deviation principle in $\mathcal{M}_1(\CC)$.
\end{enumerate}
We use inverse stereographic projection because, as $\mathcal{M}_1(\mathcal{S}^2)$ is compact, a weak large deviation principle is equivalent to a full large deviation principle, without proving exponential tightness. \cite[Lemma 1.2.18]{dembozeitouni}

In Section 2 we introduce the objects that will be studied in the article.
In Section 3 we give a dettailed proof of the result for Kac polynomials following the steps given above. Section 4 is about elliptic polynomials. As the proof is nearly the same, we focus on what should be changed to import the proof from the previous section. In Section 5 we prove the general result that was originally proved by Zeitouni and Zelditch in \cite{zeitounizelditch} and we extend it for real Gaussian coefficients.

In contrast, the article \cite{zeitounizelditch} has a more geometric and intrisinc approach. The scalar product \eqref{produitscal} is related to a notion of curvature on $\CC \PP^1$. The zeros are seen as elements of $\CC \PP^1$ and the rate function is expressed in terms of Green function and Green energy associated to this geometric setup.

\section{Definitions and notations}
We give some definitions that will be useful in the article.

\begin{definition}[Logarithmic potential, logarithmic energy]
We call the logarithmic potential of a measure $\mu \in \mathcal{M}_1(\CC)$ the function:
$$U^{\mu}: z \mapsto \int -\log |z-w|d\mu(w).$$
We also define the logarithmic energy of a measure $\mu \in \mathcal{M}_1(\CC)$:
$$\mathcal{E}(\mu)= -\iint \log|z-w|d\mu(z)d\mu(w)$$
and $J: \mathcal{M}_1(\CC) \rightarrow \RR\cup \{\infty\}$ defined by
$$J(\mu)= \sup_{z\in S}\int \log|z-w|\mu(w)$$
	where $S$ is the unit circle in $\CC$.
\end{definition}

\begin{definition}[Discrete logarithmic energy.]
	Let $\mu_n=\sum_{i=1}^{n} \delta_{z_i}$ then we write:
	\begin{equation}
	\mathcal{E}_{\neq}(\mu_n)= -\frac{1}{n^2} \sum_{i\neq j} \log |z_i-z_j|= -\int_{\neq} \log|z-w|d\mu_n(z)d\mu_n(w)
	\end{equation}
	where $\displaystyle \int_{\neq}$ stands for the off-diagonal integral .
\end{definition}
We will use the same notation for measures on $\CC$ or on $\mathcal{S}^2$.

Let us define now the inverse stereographic projection that will be the key tool in this article.
\begin{definition}[Inverse stereographic projection] \label{stereo}
	Let $\mathcal{S}^2$ be the sphere in $\RR^3$ of center $(0,0,1/2)$ and radius $1/2$. We call the point $N=(0,0,1)$ the north pole. Let $T: \CC \rightarrow \mathcal{S}^2$ the inverse stereographic projection $$T(z)= \left(\frac{\Re(z)}{1+|z|^2}, \frac{\Im(z)}{1+|z|^2}, \frac{|z|^2}{1+|z|^2}\right).$$
We have the following relations, valid for any $z$ and $w$ in $\CC$:
\begin{equation}  \label{relations}	 |z-w|^2= \frac{|T(z)-T(w)|^2}{(1-|T(z)|^2)(1-|T(w)|^2)}
\end{equation}
\begin{equation} \label{relation2}
	1-|T(z)|^2= \frac{1}{1+|z|^2}
\end{equation}	
where if $x\in \RR^3$, $|x|$ is its Euclidean norm and when $z$ is a complex number, $|z|$ is its modulus. The same notation holds for the norm in $\CC$ and the norm in $\RR^3$.
\end{definition}
The first relation can be found in \cite[lemma $3.4.2$]{ashnovinger}, and the second relation is obtained from the first one by squaring, taking the limit as $w$ tends to infinity and using the Pythagorean theorem.

\begin{figure}[!h]
	\centering
	\includegraphics[scale=1]{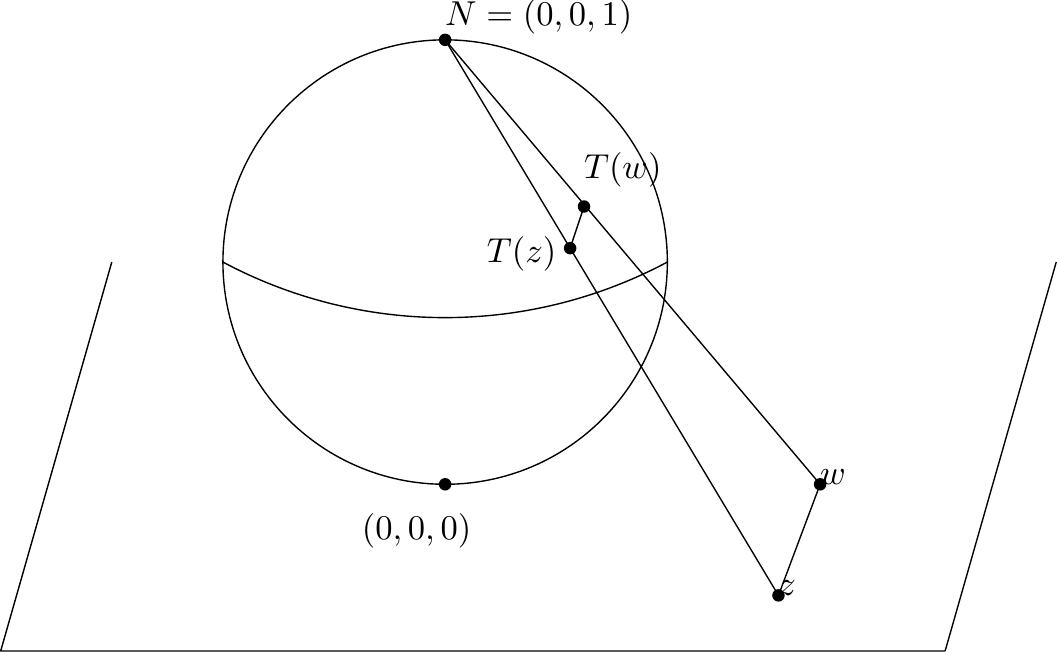}
	\caption{\label{proj} Inverse stereographic projection.}
\end{figure}

To avoid confusions between what lies in $\CC$ and what lies in $\RR$, we will only use the letters $z,w$ for complex numbers and the letters $x,y$ for vectors in $\RR^3$.

\begin{definition}[Push-forward of the objects on the sphere]
We define the push-forward by $T$ of the empirical measure:	\[\bar{\mu_n} = T^*\mu_n= \frac{1}{n} \sum_{i=1}^{n} \delta_{T(z_i)} .\]
Let $\mu \in  \mathcal{M}_1(\mathcal{S}^2)$, we call its logarithmic potential on the sphere the function: $$U^{\mu}_{\mathcal{S}^2}(x)=\int- \log |x-y|d\mu(y).$$ $U^{\mu}_{\mathcal{S}^2}$ takes its values in $[-\infty, \infty)$.
We define on $\mathcal{M}_1(\mathcal{S}^2)$ the function:			
\begin{equation*}
	\mathcal{E}_{\mathcal{S}^2}(\nu)= - \iint \log |x-y|d\nu(x)d\nu(y).
\end{equation*}
\end{definition}
The function $U^{\mu}_{\mathcal{S}^2}$ is called logarithmic potential. It inherits its name from $U^{\mu}$ as it is the analog formula on the sphere. The name logarithmic potential is not really appropriate as this notion is already defined on the sphere in potential theory, but it is convenient as the formulas are the same.
\section{Large deviations for Kac polynomials}

This section deals with the Coulomb gases \eqref{kaccomplexe} and \eqref{kacreel}. We prove Theorems \ref{LDPKac} and \ref{LDPKacreel}.

\subsection{Step 1: Distribution of the roots}
\begin{theorem}[Distribution of the roots in for the complex case]
Let $P_n= \sum_{k=0}^{n} a_k X^k$, the law of $(z_1, \dots, z_n)$ is absolutely continuous with respect to the Lebesgue measure on $\CC^n$ with density:
$$\frac{1}{Z_n}\frac{\prod_{i<j}|z_i-z_j|^2}{(\int\prod_{i=1}^N |z-z_i|^2 d\nu_S)^{n+1}}=\frac{1}{Z_n}\exp \left( -\beta_n \left[  \mathcal{E}_{\neq}(\mu_n) + \frac{n+1}{n^2}\log \int \prod_{i=1}^{n}|z-z_i|^2d\nu(z) \right]\right)$$
where $Z_n$ is a normalizing constant.
\end{theorem}

\begin{proof}
	Let $p(z)= z^n + b_{n-1}z^{n-1}+ \dots + b_0 = \prod\limits_{i=1}^n (z-z_i)$. Then the transformation
\begin{equation*}
\begin{array}{ccccc}
	F & : & \CC^{n} & \to & \CC^{n} \\
	& & (z_1,\dots,z_n) & \mapsto & (b_0, \dots , b_{n-1}) \\
\end{array}
\end{equation*}
has Jacobian determinant $\prod_{i<j} |z_i-z_j|^2$ \cite[Lemma 1.1.1]{houghtkris}.
We compute the law of the random vector $(z_1,\dots,z_n,a_n)$.
The density of the law of $(a_0,\dots, a_n)$ is $$\frac{1}{\pi^{n+1}}e^{-\sum_{k=0}^{n}|a_k|^2}$$
We consider now the change of variables:
\begin{equation*}
\begin{array}{ccccc}
G & : & \CC^{n+1} & \to & \CC^{n+1} \\
& & (z_1,\dots,z_n,a_n) & \mapsto & (a_0,\dots,a_{n-1},a_n) \\
\end{array}
\end{equation*}
whose Jacobian determinant is $|a_n|^{2n}\prod_{i<j}|z_i-z_j|^2$, as $b_i=a_i/a_n$. 
Hence, the law of $(z_1,\dots,z_n,a_n)$ is absolutely continuous with respect to the Lebesgue measure on $\CC^{n+1}$. We want to rewrite the density of the random vector $(a_0,\dots,a_n)$ with the new variables $(z_1,\dots,z_n,a_n)$. We notice that if $P=\sum_{k=0}^{n}a_k X^k$ then:
\begin{equation}\label{eq:trick}
\sum_{k=0}^{n} |a_k|^2= \int |P(z)|^2d\nu_{S}(z)=\int |a_n|^2 \prod_{k=1}^n|z-z_k|^2d\nu_{S}(z)
\end{equation}
	where $\nu_S$ is the uniform probability measure on the unit circle of $\CC$. This relation comes from the fact that the canonical basis of $\CC[X]$ is orthonormal for the scalar product \eqref{psKac}.
The density of the law of $(z_1,\dots,z_n,a_n)$ is:
	
	\[\frac{|a_n|^{2n}\prod_{i<k}|z_i-z_j|^2}{\pi^n}\exp \left(-|a_n|^2\int  \prod_{k=1}^n|z-z_k|^2d\nu_{S}(z)\right).\]
We only have to integrate in the variable $a_n$ to obtain the law of $(z_1,\dots, z_n)$.
\end{proof}

\begin{theorem}[Distribution of the roots in the real case.] \label{the:rootsexpo}
The distribution of the random vector $(z_1, \dots , z_n)$ of the roots of $P_n$ in the real case is given by:
	
\begin{equation*}
\sum_{k=0}^{\lfloor n/2 \rfloor} \frac{2^k \Gamma(\frac{n+1}{2})}{k!(n-2k)!\pi^{(n-1)/2}} \frac{ \prod_{i<j}|z_i-z_j|}{ (\int\prod_{i=1}^n |z-z_i|^2 d\nu_S)^{(n+1)/2}} d\ell_{n,k}(z_1,\dots,z_n).
\end{equation*}
This law can be re-written as: 
\begin{equation*}
\sum_{k=0}^{\lfloor n/2 \rfloor} \frac{1}{Z_{n,k}}  \exp \left(-\beta_n H(z_1,\dots,z_n)\right) d\ell_{n,k}(z_1,\dots,z_n)
\end{equation*}
where \begin{equation*} Z_{n,k}=\frac{k!(n-2k)!\pi^{\frac{n+1}{2}}}{2^k \Gamma(\frac{n+1}{2})}.
\end{equation*}
\end{theorem}

The density in the real case is nearly the same as in the complex case, except for the factor $1/2$ in the exponent. In the complex case, the vector of the zeros is an element of  $\RR^{2n}$ while in the real case, we can see the zeros as an element of $\RR^{n}$.

\begin{proof}
	As the random vector $(a_0, \dots , a_n)$ has a joint distribution: $$\frac{1}{(\pi)^{\frac{n+1}{2}}}\exp \left( -\sum_{k=0}^{n}|a_k|^2\right) da_0 \dots da_n. $$ we can use Zaporozhets' computation \cite{zaporozhets} in order to express the distribution of $(z_1, \dots , z_n, a_n)$. We use again the relation \eqref{eq:trick} in order to simplify the expression of the distribution of the roots and we obtain:

	\begin{equation*}
	\sum_{k=0}^{\lfloor n/2 \rfloor} \frac{2^k}{k!(n-2k)!\pi^{\frac{n+1}{2}}}  |a_n|^n \prod_{i<j}|z_i-z_j| e^{-\displaystyle\int |a_n|^2 \prod_{i=1}^n|z-z_i|^2d\nu_{S}(z)}  d\ell_{n,k}(z_1,\dots,z_n) d\ell_{\CC}(a_n).
	\end{equation*}
We integrate with respect to $a_n$, which ends the proof.
\end{proof}

\begin{remark}[Symmetries of the problem]
	It is easy to check that the law of the zeros is invariant under rotation as a Gaussian vector is invariance by rotation. The distribution of the zeros is also invariant under the mapping $z \mapsto 1/z$. This comes from the fact that $(a_0,\dots, a_n)$ has the same distribution as $(a_n,\dots,a_0)$, but if we call $z_1,\dots,z_n$ the zeros of $\sum_{k=0}^{n}a_k X^k$, then the zeros of $\sum_{k=0}^{n}a_{n-k}X^k$ are $1/z_1, \dots,1/z_n$.
\end{remark}

\begin{proposition}[Uniform control of $Z_{n,k}$ for $\beta_n=n^2$]\label{limits}
Let $Z_{n,k}$ given in \ref{the:rootsexpo}, we have:
\begin{equation*}
\lim_{n\to \infty} \sup_{k \in \{1,\dots,n\}} \frac{1}{n^2} |\log Z_{n,k} |= 0
\end{equation*}
which implies that
\begin{equation*}
\lim_{n\to\infty} \frac{1}{\beta_n} \log \max_{k} Z_{n,k} = \lim_{n\to\infty} \frac{1}{\beta_n} \log \min_{k} Z_{n,k}= 0.
\end{equation*}
\end{proposition}
\begin{proof}[Proof of Proposition \ref{limits}]
Using the triangular inequality and bounding $k$ by $n$ gives:
\begin{multline*}
\frac{1}{n^2}|\log Z_{n,k} | \leq \frac{1}{n^2} \log n! + \frac{1}{n^2}\log n! + \frac{n+1}{2n^2} \log \pi + \frac{n}{n^2}\log 2 + \frac{1}{n^2} \log \Gamma (\frac{n+1}{2})
\end{multline*}
The upper bound is uniform in $k$ and tends to $0$ as $n$ goes to infinity, which proves the result.
\end{proof}
This control over the $Z_{n,k}$ constants is very important. This is the reason why we are able to prove large deviations in the real case. For general $\beta_n$, we cannot prove lage deviations without assuming that those limits exist and are equal, not necessarily to zero. This will become clearer in Section \ref{getridofZ}. See \cite{benarouszeitouni} and \cite{ghoshzeitouni} for similar results.

\subsection{Step 2: Large deviations on the unit sphere}
In order to prove the large deviation principles, we are going to use a compactification method introduced in \cite{hardy}. When the potential does not grow faster than a logarithm at infinity, the standard proofs of large deviations principles do not hold. More precisely, exponential tightness of the sequence of measures cannot be proved using the standard techniques presented in \cite{benarousguionnet}, \cite{benarouszeitouni}, \cite{hiaipetz}. The gas we are studying is also weakly confining as the confinement term grows at infinity like $\log(1+|z|^2)$ in each variable.

Using the inverse stereographic projection $T$ \eqref{stereo} we will push the problem on the sphere $\mathcal{S}^2$ in $\RR^3$. As the sphere is a compact set, it is sufficient to prove a weak large deviation principle instead of a full one.
 
\begin{remark}[Push Forward]
In this article, we will use the notation
$T^*\mu$
for the push-forward of the measure $\mu$ by the function $T$. 
\end{remark}

\begin{definition}[Measure on $\mathcal{S}^2$]
	We call $L_{\CC}$ the push-forward of the Lebesgue measure of $\CC$ on $\mathcal{S}^2$ by $T$ and $L_{\RR}$ the push-forward of the Lebesgue measure on $\RR$ by $T$, where $\RR$ is seen as a subspace of $\CC$. We will use the notation:
	$$ dL_{n,k}(x_1,\dots,x_n)= dL_{\RR}(x_1) \dots  dL_{\RR}(x_{n-2k})  dL_{\CC}(x_{n-k}) \dots  dL_{\CC}(x_n).$$
\end{definition}

\begin{proposition}[Pushing the complex case on the sphere]\label{pushcomplexkac}
	Let $(z_1,\dots,z_n)$ be the zeros of $P_n$ in the complex case, then the law of $(T(z_1),\dots,T(z_n))$ is absolutely continuous with respect to the push forward by $T$ of the Lebesgue measure on $\CC$ with density:
	
	$$\frac{\prod_{i<j} |x_i-x_j|^2}{(\int \prod_{j=1}^n |x-x_j|^2 2^n dT^*\nu_S(x))^{n+1}}  \times \prod_{i=1}^{n} (1-|x_i|^2)^2.$$
We call $\kappa_n$ the finite measure:
\begin{equation*}
\kappa_n=\prod_{i=1}^{n} (1-|x_i|^2)^2 dL_{\CC}(x_1)\dots dL_{\CC}(x_n).
\end{equation*}
This law can be written in the form:
\begin{equation}\label{gazsphere}
\frac{1}{Z_n}\exp \left(-\beta_n \left[ \mathcal{E}_{\neq}(\bar{\mu}_n) + \frac{n+1}{n^2}\log \int \prod_{j=1}^n |x-x_j|^2 2^n dT^*\nu_S(x)\right]\right) d\kappa_n.
\end{equation}
\end{proposition}

\begin{remark}[Identification of the uniform measure on $\mathcal{S}^2$]
	The measure $(1-|x|^2)^2 dL_{\RR}(x)$ on $\mathcal{S}^2$ is proportional to the uniform measure on the sphere. Indeed if we push forward this measure by the stereographic projection $T^{-1}$ we obtain the measure $\frac{1}{(1+|z|^2)^2}d\ell_{\CC}(z)$ which is proportional to the complex Cauchy measure, which is known to be the projection of the uniform measure on the sphere.
\end{remark}

\begin{proof}[Proof of proposition \ref{pushcomplexkac}]
 We will now push the zeros of $P_n$ on the sphere $\mathcal{S}^2$. We compute the law of the vector $(T(z_1), \dots, T(z_n))$. We use the relations \eqref{relations} to obtain:
\begin{equation*}
\prod_{i<j}|z_i-z_j|^2  =\prod_{i<j} \frac{|T(z_i)-T(z_j)|^2}{(1-|T(z_i)|^2)(1-|T(z_j)|^2)}
\end{equation*}
 and that:
 \begin{equation*}
\left(\int\prod_{i=1}^N |z-z_i|^2 d\nu_S(z)\right)^{n+1}=\left(\dfrac{\prod_{i=1}^{n}(1-|T(z_i)|^2)}{\displaystyle{\int}\prod_{i=1}^n |T(z)-T(z_i)|^2 (1-|T(z)|^2)^{-n} d\nu_S(z)}\right)^{n+1}.
 \end{equation*}
We notice that on the unit circle of $\CC$, the function $z \mapsto |T(z)|^2$ is constant equal to $1/2$, so we can write:
 \begin{equation}
 \frac{\prod_{i<j}|z_i-z_j|^2}{(\int\prod_{i=1}^N |z-z_i|^2 d\nu_S(z))^{n+1}}  = \frac{\prod_{i<j} |T(z_i)-T(z_j)|^2}{(\int \prod_{i=1}^n |T(z)-T(z_i)|^2 2^n d\nu_S(z))^{n+1}}  \times \prod_{i=1}^{n} (1-|T(z_i)|^2)^2.
\end{equation}
 \end{proof}

\begin{proposition}[Pushing the real case on the sphere]\label{pushreelkac}
	Let $(z_1,\dots,z_n)$ be the zeros of $P_n$ in the real case, then the law of $(T(z_1),\dots,T(z_n))$ is:
	$$\sum_{k=0}^{\lfloor n/2 \rfloor} \frac{1}{Z_{n,k}} \frac{ \prod_{i<j}|x_i-x_j|\times \prod_{i=1}^{n} (1-|x_i|^2)}{ (\int\prod_{i=1}^N |x-x_i|^2 2^ndT^*\nu_S)^{(n+1)/2}}  dL_{n,k}(x_1,\dots,x_n)$$
We call $\rho_{n,k}$ the finite measure: $$d\rho_{n,k}=\prod_{i=1}^{n} (1-|x_i|^2)dL_{n,k}(x_1,\dots,x_n).$$
This law can be written:
\begin{equation}\label{gazsphere2}
\sum_{k=0}^{\lfloor n/2 \rfloor} \frac{1}{Z_{n,k}}  \exp \left( -\beta_n\left[ \frac{1}{2}	\mathcal{E}_{\neq}(\mu_n)  -\frac{n+1}{2n^2} \log \int\prod_{i=1}^N |x-x_i|^2 2^n dT^*\nu_S(x) \right]\right) d\rho_{n,k}.
\end{equation}
\end{proposition}
The proof of this proposition is exactly the same as the one of Proposition \ref{pushcomplexkac}.

The measure $T^*\nu_S$ is the uniform measure on the equator of the sphere $\mathcal{S}^2$. Seeing those measures on the sphere emphasizes the symmetries of the problem as the invariance with respect to inversion corresponds to the exchange of north and south pole of the sphere.

$\rho_n$ is a finite measure. As $\rho_n$ is a product measure, we only have to see that every measure is finite. There are two types of measures in this product: 
$$(1-|x|^2)dL_{\RR}(x) = \frac{1}{1+|x|^2}dx$$ which is finite on $\RR$
and
$$(1-|x|^2)(1-|y|^2)dL_{\CC}(x)= \frac{1}{(1+|z|^2)(1+|\bar{z}|^2)}d\ell_{\CC}(z)$$ where $x$ and $y$ are the inverse stereographic projection of 

We now state the large deviation principle on the sphere $\mathcal{S}^2$.

\begin{definition}[Rate function in $\mathcal{M}_1(\mathcal{S}^2)$]
For any measure $\nu \in \mathcal{M}_1(\mathcal{S}^2)$ we define:
\begin{align}
J_{\mathcal{S}^2}(\nu) &= \sup_{x\in T(S)} \int \log |x-y|^2d\nu(y)+\log 2 \\ I_{\mathcal{S}^2}(\nu)& = \mathcal{E}(\nu) + J_{\mathcal{S}^2}(\nu)
\end{align}
\end{definition}

\begin{proposition}[Complex Kac case on the sphere] \label{thmcomplexe}
Let $\bar{\mu}_n$ be the empirical measure of the gas \eqref{gazsphere},
then $(\bar{\mu}_n)_{n\in \NN}$ satisfies a large deviation principle in $\mathcal{M}_1(\mathcal{S}^2)$ with the weak topology, speed $\beta_n$ and good rate function $I_{\mathcal{S}^2}$. This means that for any Borel set $A$ in $\mathcal{M}_1(\mathcal{S}^2)$
\begin{equation*}
-\inf_{\mathrm{Int}A} I_{\mathcal{S}^2} \leq \varliminf_{n\to\infty} \frac{1}{\beta_n} \log \PP(\bar{\mu}_n \in A) \leq \varlimsup_{n\to\infty} \frac{1}{\beta_n} \log \PP(\bar{\mu}_n \in A) \leq -\inf_{\mathrm{Clo}A} I_{\mathcal{S}^2}.
\end{equation*}
\end{proposition}

\begin{proposition}[Real Kac case on the sphere] \label{thmreel}
Let $\bar{\mu}_n$ be the empirical measure of the gas \eqref{gazsphere2},
then $(\bar{\mu}_n)_{n\in \NN}$ satisfies a large deviation principle in $\mathcal{M}_1(\mathcal{S}^2)$ with the weak topology, speed $\beta_n$ and good rate function $\tilde{I}_{\mathcal{S}^2}$ where:
	
	$$	\tilde{I}_{\mathcal{S}^2}(\mu)=\begin{cases}
	\frac{1}{2}I_{\mathcal{S}^2}(\mu)& \text{if $T^{-1*}\mu$ is invariant under }z \mapsto \bar{z} \\\infty	&   \text{otherwise.} 
	\end{cases} 
	$$
This means that for any Borel set $A$ in $\mathcal{M}_1(\mathcal{S}^2)$
\begin{equation*}
-\inf_{\mathrm{Int}A} \tilde{I}_{\mathcal{S}^2} \leq \varliminf_{n\to\infty} \frac{1}{\beta_n} \log \PP(\bar{\mu}_n \in A) \leq \varlimsup_{n\to\infty} \frac{1}{\beta_n} \log \PP(\bar{\mu}_n \in A) \leq -\inf_{\mathrm{Clo}A} \tilde{I}_{\mathcal{S}^2}.
\end{equation*}
\end{proposition}

\subsection{Step 3: Proof of the Large Deviations Principles}
We now prove the Proposition \ref{thmcomplexe} and Proposition \ref{thmreel}. We start by studying the rate function, then we prove the lower and upper bound for the gas without the normalizing constants $Z_n$ and $Z_{n,k}$. Finally, we obtain the full large deviation principle.

\subsubsection{Study of the rate function on the sphere}
Next proposition is the key of all the rest of the work and comes from \cite[Lemma 26]{zeitounizelditch}.
\begin{proposition}[Rate function $I_{\mathcal{S}^2}$] \label{ratefunction}
	\item
	1)$J_{\mathcal{S}^2}$ is continuous for the weak topology of $\mathcal{M}_1(\mathcal{S}^2)$ and is bounded.
	\item
	2)$I_\mathcal{S}^2$ is well defined on $\mathcal{M}_1(\mathcal{S}^2)$, takes its values in $[0,\infty]$.
	\item
	3) $I_{\mathcal{S}^2}$ is lower semi-continuous.
	\item
	4) $I_{\mathcal{S}^2}$ is strictly convex.
\end{proposition}

\begin{proof}[Proof of Proposition \ref{ratefunction}]
	First, we notice that since $\mathcal{S}^2$ is a compact set in $\RR$, the function $(x,y)\mapsto \log |x-y|$ is bounded above on $\mathcal{S}^2 \times \mathcal{S}^2$. Hence, the function $\mathcal{E}_{\mathcal{S}^2}$ is bounded from below and $J_{\mathcal{S}^2}$ is bounded from above. We cannot conclude yet that the function $I_{\mathcal{S}^2}$ is well defined. We will see that since $J_{\mathcal{S}^2}$ is continuous on the compact $\mathcal{M}_1(\mathcal{S}^2)$, it is bounded and $I_{\mathcal{S}^2}$ is well defined and takes its values in $(-\infty, \infty ]$.
	
	Fix $M \in \RR$ and define $\log_M(x)= \log(x)\vee (-M)$.  The function $(x,y)\mapsto \log_M|x-y|$ is continuous on $\mathcal{S}^2 \times \mathcal{S}^2$. We define:
	
	\begin{align*}
	U^{\nu}_M(x)&= -\int \log_M|x-y|d\nu(y) \\
	\mathcal{E}_{\mathcal{S}^2}^M(\nu)&= - \iint \log_M |x-y|d\nu(x)d\nu(y) \\
	J_{\mathcal{S}^2}^M(\nu) &= \sup_{x\in T(S)} \int \log_M |x-y|^2d\nu(y)+\log 2\\ 
	I_{\mathcal{S}^2}^M(\nu) &= \  \mathcal{E}_{\mathcal{S}^2}^M(\nu) + J_{\mathcal{S}^2}(\nu).
	\end{align*}
We will now prove the upper semi-continuity and lower semi-continuity of the function $J_{\mathcal{S}^2}$.

	\begin{itemize}
		\item \textbf{Upper semi-continuity}
	\end{itemize}
	
	The map $W:(x,\mu) \mapsto U^{\mu}_M(x)$ is continuous because the function $(x,y)\mapsto \log_M|x-y|$ is uniformly continuous. As $T(S)$ is compact, $J_{\mathcal{S}^2}^M$ is also continuous. Indeed, if $\mu_n \rightarrow \mu$ weakly, then there exist a sequence $(x_n)_{n \in \NN}$ such that for any $n$,  $W(x_n,\mu_n)=J_{\mathcal{S}^2}^M(\mu_n)$. Let $c$ be an accumulation point of the sequence $(J_{\mathcal{S}^2}^M(\mu_n))_{n\in \NN}$. One can extract a convergving subsequence of $(x_n)_{n \in \NN}$ and call $x_{\infty}$ its limit. Taking the limit of the inequality 
	\begin{equation*}
	W(x_n,\mu_n) \geq W(x ,\mu_n)
	\end{equation*} 
	for any fixed $x$ shows that $c = J_{\mathcal{S}^2}^M(\mu)$. Hence, $J_{\mathcal{S}^2}^M$ is continuous.
	Now let $\mu_n \rightarrow \mu$ weakly in $\mathcal{M}_1(\mathcal{S}^2)$, we have:
	
	$$J_{\mathcal{S}^2}(\mu_n) \leq J_{\mathcal{S}^2}^M(\mu_n) \xrightarrow[n \rightarrow \infty]{} J_{\mathcal{S}^2}^M (\mu)$$
If we take the limit superior of the last inequality we obtain: 
	$$\varlimsup_{n\to\infty} J_{\mathcal{S}^2}(\mu_n) \leq J_{\mathcal{S}^2}^M (\mu)$$
	To conclude, we want to let $M$ go to infinity, but we have to justify the exchange between the limit and the supremum. In order to do that, we use a short lemma given below:
	
	\begin{lemma}
		Let $(f_M)_{M\in \RR^+}$ be a decreasing sequence of continuous functions defined on a compact set $K$, converging point-wise towards a function $f$. Then we have:
		$$ \lim_{M\to\infty} \sup_{z\in K}f_M(z)=\sup_{z\in K} f(z).$$
	\end{lemma}
	\begin{proof}
It is easy to show that the function $M \mapsto \sup_{z\in K}f_M(z)$ is decreasing and is bounded below by $\sup_{z\in K} f(z)$. Hence we obtain $\lim\limits_{M \rightarrow \infty} \sup_{z\in K}f_M(z) \geq \sup_{z\in K} f(z)$. To prove the other inequality, fix $\varepsilon >0$, then, as $f_M(x)$ decreases towards $f(x)$, we have:
\[ \forall x \in K, \exists M_x >0 \text{ such that } \forall M\geq M_x, f(x) \geq f_M(x) \geq f(x)+ \varepsilon.\]
As $f_{M_x}$ is continuous at $x$, there is an $\delta_x$ such that for all $y \in B(x,\delta_x)$ : $f(x)-\varepsilon \leq f_{M_x}(y) \leq f(x) + 2 \varepsilon$. As the sequence $f_M$ is decreasing, this relation is also satisfied for all $M \geq M_x$. As $K$ is compact, we can extract a finite family $\{x_1 ,\dots ,x_p\}$ such that $K \subset \cup B(x_i, \delta_i)$. We set $M_{\infty}= \max_{i=1 \dots p} M_{x_i}$ so we have:
$$ \forall M \geq M_{\infty}, \forall y \in K, \exists i \in \{1, \dots, p\} \mid f(x_i)- \varepsilon \leq f_M(y) \leq f(x_i)+ 2 \varepsilon.$$ This last statement implies that $\sup_K f_M(z) \leq \sup_K f(z) + 2 \varepsilon$ and ends the proof.  
\end{proof}
We apply the lemma and we end the proof of the upper semi-continuity:	
\[	\varlimsup_{n\to\infty} J_{\mathcal{S}^2}(\mu_n) \leq J_{\mathcal{S}^2}(\mu). \]
\begin{itemize}
		\item \textbf{Lower semi-continuity}
\end{itemize}
This is where the notion of non-thinness is involved. We will use the fact that the support of $T^* \nu_S$ is non thin at all its points.
Suppose that $\mu_n \xrightarrow[n\rightarrow \infty]{}\mu$. We want to show that $\varliminf_{n\to\infty} J_{\mathcal{S}^2}(\mu_n) \geq J_{\mathcal{S}^2}(\mu)$. We know that $J_{\mathcal{S}^2}(\mu) < \infty$. If $J_{\mathcal{S}^2}(\mu) = - \infty$ then there is nothing to prove. For any $\varepsilon > 0$ we introduce the set:
\[	A_{\varepsilon}= \{ x \in \mathcal{S}^2 \mid -2 U^{\mu}(x) +\log 2 \geq J_{\mathcal{S}^2}(\mu) - \varepsilon \}.\]
For any $\varepsilon$, $A_{\varepsilon}$ is closed by upper semi-continuity of $-U^{\mu}$. Let $x_0$ be a point where $-U^{\mu}$ reaches its maximum on the equator $T(S)$. We want to find a measure of positive and finite mass $\nu$ supported on $A_{\varepsilon}\cap T(S)$ such that $U^{\nu}$ is a continuous function, for any $\varepsilon$. We can find such a measure as soon as the capacity of $A_{\varepsilon}\cap T(S)$ is positive \cite[Chapter 1, Corollary 6.11]{safftotik}. If for some $\varepsilon_0 > 0$ the set $A_{\varepsilon}\cap T(S)$ had zero capacity, it would be thin at any point \cite[Theorem 3.8.2 p79]{ransford1995potential}. By the definition of the set $A_{\varepsilon_0}$, the complement of $A_{\varepsilon}\cap T(S)$ in $T(S)$ is thin at $x_0$. Then, as the union of two thin sets at $x_0$ is thin at $x_0$, $T(S)$ is thin at $x_0$. This is absurd as the equator is non-thin at all its points (as connected set, \cite[Theorem 3.8.3]{ransford1995potential}).
	
Now that we obtained the existence of $\nu$, we can end the proof. Thanks to Fubini's theorem, we have:
\begin{align*}
	\lim\limits_{n\rightarrow \infty} \int -U^{\mu_n}(x) d\nu(x)  & = \lim\limits_{n\rightarrow \infty} \int- U^{\nu}(x) d\mu_n(x)\\
	& = \int- U^{\nu}(x) d\mu(x) \\
	& = \int -U^{\mu}(x) d\nu(x).
\end{align*} 
Since the support of $\nu$ is included in $A_{\varepsilon}$, we have:
	$$ \int \varliminf_{n\to\infty} J_{\mathcal{S}^2}(\mu_n) d\nu(x) \geq \lim_{n\to\infty} \int -2U^{\mu_n}(x)d\nu(x) +\log 2 \geq \int \left( J_{\mathcal{S}^2}(\mu) - \varepsilon \right)d\nu(x). $$
	We end the proof by noticing that $\nu$ has positive mass and that $\varepsilon$ is arbitrary.
	
	\begin{itemize}
		\item \textbf{Conclusion}
	\end{itemize}
	
As $J_{\mathcal{S}^2}$ is continuous on the compact $\mathcal{M}_1(\mathcal{S}^2)$, it is bounded and the function $I_{\mathcal{S}^2}$ is well defined.
	
	For each fixed $x$, the functions $\mu \mapsto U^{\mu}_M(x)$ and $\mu \mapsto \mathcal{E}_{\mathcal{S}^2}^M(\mu)$ are continuous on $\mathcal{M}_1(\mathcal{S}^2)$. Since $\mathcal{E}_{\mathcal{S}^2}= \inf_{M}\mathcal{E}_{\mathcal{S}^2}^M$, $\mathcal{E}_{\mathcal{S}^2}$ is a lower semi-continuous function. $I_{\mathcal{S}^2}$ is lower semi-continuous as the sum of a continuous and lower semi-continuous functions.
	
	It is well known that the classical interaction energy $\mathcal{E}$ is a convex function \cite{deift}, \cite[Proposition $5.3.2$]{hiaipetz}. For an measure with finite energy, we can write: $$\mathcal{E}_{\mathcal{S}^2}(\mu)= \mathcal{E}((T^{-1})^*\mu) + \int \log(1+|z|^2)d(T^{-1})^*\mu.$$ Since the function $\nu \mapsto \displaystyle \int \log|1+|z|^2|d\nu$ is linear we have the convexity of $\mathcal{E}_{\mathcal{S}^2}$. On the other hand, $J_{\mathcal{S}^2}$ is the supremum of linear functions so is convex.
	
\end{proof}
\subsubsection{Large deviations upper bound}

We will prove the upper bound for the non-normalized measures, the normalizing constant will be treated once we have both upper and lower bound.
\begin{itemize}
	\item \textbf{Bernstein-Markov inequality}
\end{itemize}

We need to prove a Bernstein-Markov inequality for the measures $\nu_S$ and $T^*\nu_S$.

\begin{theorem}[Bernstein-Markov for $\nu_S$]
	Let $N\in \NN$, then for all $P \in \CC_N[X]$ we have:
$$ \sup_{z \in S^1}|P(z)| \leq \sqrt{N} \|P\|_{L^2}$$ where $\|P\|_{L^2}^2 = \int |P(z)|^2d\nu_S(z)$. In particular, for all $\varepsilon > 0$,  there is a constant $C_{\varepsilon}$ such that for all $P \in \CC_N[X]$ we have:
$$ \sup_{z \in S^1}|P(z)| \leq C_{\varepsilon} e^{\varepsilon N}  \|P\|_{L^2}.$$ 
\end{theorem}

\begin{proof}
We start from the following identity:
\[P_n(z)= \int P_n(w) \sum_{k=0}^n z^k \bar{w}^k d\nu_S(w).\]
Then by the Schwarz inequality we obtain:
\[\sup_{z \in S} |P_n(z)| \leq \sup_{z \in S} \sqrt{\sum_{k=0}^n |z|^{2k} } \|P_n\|_{L^2(\nu_S)} \leq \sqrt{n}\|P_n\|_{L^2(\nu_S)}.\]
\end{proof}

\begin{lemma} \label{bernmarko}
	For all $\varepsilon > 0$, there exists an integer $N_0$ such that for all $n>N_0$ we have:
\begin{equation*}
 \frac{1}{n} \log  \int e^{2n\int_{\CC}\log|z-w|d\mu_n(w)}d\nu_S(z) \geq 2\sup_{z \in S} \int_{\CC} \log|z-w|d\mu_n(w) - \varepsilon.
\end{equation*}
\end{lemma}

\begin{proof}[Proof of Lemma \ref{bernmarko}]
	\[\int e^{2n\int_{\CC}\log|z-w|d\mu_n(w)}d\nu_S(z)= \int \prod_{i=1}^{n} |z-z_i|^2 d\nu_S(z).\]
Thanks to the Bernstein-Markov inequality, we have:
\[\left( \int \prod_{i=1}^n |z-z_i|^2 d\nu_S(z)\right) ^{1/n} \geq \left( \frac{\sup_{z \in S}\prod_{i=1}^n |z-z_i|^2}{C_{\varepsilon/2} e^{-n\varepsilon/2 }}\right)^{1/n}.\]
We take the logarithm of this expression to obtain:
\[\frac{1}{n} \log  \int  e^{2n\int_{\CC}\log|z-w|d\mu_n(w)}d\nu_S(z) \geq 2\sup_{z\in S} \int_{\CC} \log|z-w|d\mu_n(w) - \varepsilon/2 + \frac{\log C_\varepsilon/2}{n}\]
taking $n>N_0$ sufficiently large ends the proof of the lemma.
\end{proof}

In fact, this inequality is used to prove the large deviations upper bound on $\CC$. To prove the large deviations upper bound on $\mathcal{S}^2$, we need an analog of this inequality.

\begin{lemma}\label{bernmarko2}
For all $\varepsilon > 0$, there exists an integer $N_0$ such that for all $n>N_0$ we have:
\begin{equation*}
\frac{1}{n} \log  \int e^{2n\int_{\CC}\log|x-y|dT^*\mu_n(y)}dT^*\nu_S(x) \geq 2\sup_{x \in T(S)} \int\log|x-y|dT^*\mu_n(y) - \varepsilon.
\end{equation*}

\end{lemma}

\begin{proof}
	We start from Lemma \ref{bernmarko} and lift it up on the sphere:
\begin{align*}
\frac{1}{n} \log \int \prod_{i=1}^n \frac{|T(z)-T(z_i)|^2}{(1-|T(z)|^2)(1-|T(z_i)|^2)} & d\nu_S(z) \geq\\  \sup_{z\in S} \int_{\CC} \log& \frac{|T(z)-T(w)|^2}{(1-|T(z)|^2)(1-|T(w)|^2)}d\mu_n(w) - \varepsilon + \frac{\log C_\varepsilon}{n}.
\end{align*}
	In terms of push-foward measures we obtain:
\[\frac{1}{n} \log \int \prod_{i=1}^n |x-x_i| dT^*\nu_S(z) \geq \sup_{x\in T(S)} \int_{\CC} \log|x-x_i|dT^*\mu_n(w) - \varepsilon + \frac{\log C_\varepsilon}{n}.\]
\end{proof}

\begin{itemize}
	\item \textbf{Large deviations upper bound in the complex case}
\end{itemize}

We prove the upper bound part of the large deviation principle in the complex case. Let $\sigma \in \mathcal{M}_1(\mathcal{S}^2)$, we prove that:
$$\lim\limits_{\delta \rightarrow O} \varlimsup_{n\rightarrow \infty} \frac{1}{\beta_n}\log \PP(\bar{\mu}_n \in B(\sigma, \delta)) \leq -I_{\mathcal{S}^2}(\sigma).$$
Proving this inequality is sufficient to obtain the upper bound as $\mathcal{S}^2$ is a compact set. Indeed, a weak large deviation principle implies a full large deviation principle when we have exponential tightness, which is automatic on a compact set.

If we write $$A_1=Z_n	\PP(\bar{\mu}_n \in B(\sigma, \delta))$$ then we have for any $M>0$ and for any $\delta >0$, using Lemma \ref{bernmarko2}:
\begin{align*}
A_1  = &\int  1_{\bar{\mu}_n \in B(\sigma, \delta)}\exp \left( -\beta_n \left[  \mathcal{E}_{\neq}(\bar{\mu}_n) + \frac{n+1}{n^2}\log \int \prod_{i=1}^n |x-x_j|^2 2^n dT^*\nu_S(x)\right]\right) d\kappa_n 
\\ \leq& \int \exp \left(-\beta_n \left[  \mathcal{E}_{\neq}(\bar{\mu}_n) + \frac{n+1}{n^2}\log \int \prod_{i=1}^n |x-x_j|^2 2^n dT^*\nu_S(x)\right]\right) d\kappa_n
\\
	 \leq &  \int \exp \left(-\beta_n\left[-\frac{1}{n^2}\sum_{i\neq j} \log_M |x_i-x_j| + \frac{n+1}{n}( J_{\mathcal{S}^2}(\mu_n)  + \varepsilon ) \right] \right) d\kappa_n   
\\
	 \leq & \int \exp \left(-\beta_n\left[ \iint_{\neq} - \log_M |x-y| d\bar{\mu}_n(x)d\bar{\mu}_n(y)  +\frac{n+1}{n} (J_{\mathcal{S}^2}(\mu_n)+ \varepsilon ) \right]\right)  d\kappa_n.
\end{align*}
We now observe the following:
\begin{equation}
\frac{1}{n^2}\sum_{i\neq j} \log_M |x_i-x_j|=\iint_{\neq}  \log_M |x-y| d\bar{\mu}_n(x)d\bar{\mu}_n(y) -\frac{M}{n} = \mathcal{E}_{\mathcal{S}^2}^M(\bar{\mu}_n).
\end{equation} 
We also notice that:
\begin{equation}\label{eq:diago}
\frac{1}{\beta_n} \log \int 1 d\kappa_n  \xrightarrow[n\rightarrow \infty]{} 0. 
\end{equation}
Indeed, as $\kappa_n$ is a product of finite measures, we have $\frac{1}{n}\log \int 1 d\kappa_n = \log \int 1  d\kappa_1$.
Then, by taking the logarithm and the superior limit we obtain:
$$\varlimsup_{n\to\infty} \frac{1}{\beta_n} \log Z_n \PP(\bar{\mu}_n \in B(\sigma, \delta)) \leq \sup_{B(\sigma, \delta)} -I^M_{\mathcal{S}^2} + \varepsilon$$
As $I^M_{\mathcal{S}^2}$ is a continuous function, we have when $\delta$ converges towards $0$:
$$\lim\limits_{\delta \rightarrow 0}\varlimsup_{n\to\infty} \frac{1}{\beta_n} \log Z_n \PP(\bar{\mu}_n \in B(\sigma, \delta)) \leq  -I^M_{\mathcal{S}^2}(\sigma) + \varepsilon.$$
We end the proof of the upper bound by letting $M \rightarrow \infty$ (and using the monotone convergence theorem) and then letting $\varepsilon \rightarrow 0$.
We will get rid of the normalizing constant once we have proved the lower bound.

\begin{itemize}
	\item \textbf{Large deviations upper bound in the real case}
\end{itemize}

We prove the same bound as previously in the real case. The proof is nearly the same and we will omit what is exactly the same in the two cases. The only difference is that the law of the roots of $P_n$ in the real case is not absolutely continuous with respect to a product measure, but is a mix between such measures.

As in the complex case, we want to prove that:
$$\lim\limits_{\delta \rightarrow O} \varlimsup_{n\rightarrow \infty} \frac{1}{\beta_n}\log \PP(\bar{\mu}_n \in B(\sigma, \delta)) \leq -\frac{1}{2}I_{\mathcal{S}^2}(\sigma).$$
We begin with 
$$ \PP(\bar{\mu}_n \in B(\sigma, \delta))= \sum_{k=0}^{\lfloor n/2 \rfloor}\frac{1}{Z_{n,k}} I_{n,k,\delta}$$
where 
\begin{equation}
I_{n,k,\delta} = \int 1_{\bar{\mu}_n \in B(\sigma, \delta)} \frac{ \prod_{i<j}|x_i-x_j|}{ (\int\prod_{i=1}^n |x-x_i|^2 2^ndT^*\nu_S)^{(n+1)/2}} \rho_{n,k} . 
\end{equation}

We will prove the upper-bound for each of the $I_{n,k,\delta}$ uniformly in $k$ which will be sufficient to prove the upper bound for the non-normalized measure. The constants will be treated once we will have the full large deviation principle.

The estimates used in the complex case do no rely on the complex structures but only on ``algebraic'' inequalities so the same computations on the $I_{n,k,\delta}$ can be done.  As the formulas are the same, we obtain the same bounds.
\begin{align*}
I_{n,k,\delta}  \leq & \int 1_{\bar{\mu}_n \in B(\sigma, \delta)} \exp \left( -\beta_n\left[  \frac{1}{2}\mathcal{E}_{\neq}(\bar{\mu}_n)  +\frac{n+1}{2n^2} \log \int\prod_{i=1}^N |x-x_i|^2 2^ndT^*\nu_S(x) \right]\right) d\rho_{n,k} \\
 \leq & \int \exp \left(-\beta_n\left[  \frac{1}{2}\mathcal{E}_{\neq}(\bar{\mu}_n) +\frac{n+1}{2n^2} \log \int\prod_{i=1}^N |x-x_i|^2 2^ndT^*\nu_S(x) \right]\right) d\rho_{n,k}
\\
\leq &  \int \exp \left(-\beta_n\left[ - \frac{1}{2n^2}\sum_{i\neq j}\log_M|x_i-x_j|  +\frac{n+1}{2n} (J_{\mathcal{S}^2}(\mu_n)+  \varepsilon) \right]\right) d\rho_{n,k}    
\\
\leq & \int \exp \left(-\beta_n\left[ -\frac{1}{2}\iint_{\neq}  \log_M |x-y| d\bar{\mu}_n(x)d\bar{\mu}_n(y)  +\frac{n+1}{2n} (J_{\mathcal{S}^2}(\mu_n)+ \varepsilon ) \right] \right) d\rho_{n,k}.
\end{align*}
We need to check that 
$$ \frac{1}{\beta_n} \log \int \rho_{n,k} \xrightarrow[n \rightarrow \infty]{} 0.$$ Just like in the complex case, this is only a consequence of the fact that $\rho_n$ is a product measure.

Using again \eqref{eq:diago}, we apply logarithm to both sides of the inequality, divide by $\beta_n$ to find:
$$\varlimsup_{n\to\infty} \frac{1}{\beta_n} \log (I_{n,k,\delta}) \leq \sup_{B(\sigma, \delta)} -\tilde{I}^M_{\mathcal{S}^2} + \varepsilon.$$
By the monotone convergence theorem, we let $M\rightarrow \infty$ and then $\varepsilon \rightarrow 0$ which ends the proof of the inequality.
As the upper bound is uniform in $k$, it ends the proof of the lower bound.

\subsubsection{Large deviations lower bound.}
In this section we prove the lower bound of the large deviation principle. We notice that the rate function is the sum of a lower semi-continuous function $\mathcal{E}_{\mathcal{S}^2}$ and of the continuous function $J_{\mathcal{S}^2}$. The continuity of $J_{\mathcal{S}^2}$ allows us to treat only the logarithmic energy, which is well known in potential theory.
\begin{itemize}
	\item \textbf{Large deviations lower bound in the complex case}
\end{itemize}
Let $\sigma \in \mathcal{M}_1(\mathcal{S}^2)$, we prove that:
$$\lim\limits_{\delta \rightarrow 0} \varliminf_{n\rightarrow \infty} \frac{1}{\beta_n}\log Z_n \PP(\bar{\mu}_n \in B(\sigma, \delta)) \geq -I_{\mathcal{S}^2}(\sigma).$$
We can assume that the measure $\sigma$ satisfies $I_{\mathcal{S}^2}(\sigma)< \infty$ as if the rate function is infinite this bound holds clearly. We notice that it is equivalent to have $\mathcal{E}_{\mathcal{S}^2}<\infty$. In particular, such a measure $\sigma$ has no atom (and $\sigma(\{N\})=0$).
For any $(x_1,\dots,x_n)$, we have:
\begin{equation}\label{ineqbidon}
\frac{1}{n} \log \int \prod_{i=1}^n |x-x_j|^2 2^n dT^*\nu_S(x)\leq J_{\mathcal{S}^2}(\bar{\mu}_n).
\end{equation}
If we write:
$$A_1= Z_n	\PP(\bar{\mu}_n \in B(\sigma, \delta))$$
then we have, for any $\varepsilon >0$, if $\delta$ is small enough, using \eqref{ineqbidon}:
\begin{align*}
	 A_1 = &\int  1_{\bar{\mu}_n \in B(\sigma, \delta)}\exp \left(-\beta_n \left[  \mathcal{E}_{\neq}(\bar{\mu}_n) + \frac{n+1}{n^2}\log \int \prod_{i=1}^n |x-x_j|^2 2^n dT^*\nu_S(x)\right] \right)d\kappa_n    \\
	  \geq & \int  1_{\bar{\mu}_n \in B(\sigma, \delta)}\exp \left(-\beta_n \left[  \mathcal{E}_{\neq}(\bar{\mu}_n) + \frac{n+1}{n}J_{\mathcal{S}^2}(\bar{\mu}_n)\right]\right) d\kappa_n \\
	   \geq & \exp \left(- \beta_n \left[ \frac{n+1}{n}(J_{\mathcal{S}^2}(\sigma) + \varepsilon )\right]\right)\int  1_{\bar{\mu}_n \in B(\sigma, \delta)}\exp \left(-\beta_n   \mathcal{E}_{\neq}(\bar{\mu}_n) \right) d\kappa_n .
\end{align*}
The last inequality comes from the continuity of $J_{\mathcal{S}^2}$. To study the right integral, we will use the stereographic projection $T^{-1}$. Thanks to the relation \eqref{relations} we obtain:
$$\sum_{i \neq j} \log |x_i-x_j| = \sum_{i \neq j} \log |T^{-1}(x_i)-T^{-1}(x_j)| -(n-1) \sum_{i=1}^{n} \log (1+ |T^{-1}(x_i)|^2).$$
Pushing back the measure $\kappa_n$ leads to:
\begin{multline*}
	\int  1_{\bar{\mu}_n \in B(\sigma, \delta)}\exp\left( -\beta_n \mathcal{E}_{\neq}(\bar{\mu}_n)\right) d\kappa_n   \\ =\int 1_{\mu_n \in T^{-1}B(\sigma, \delta)} \exp \left(-\beta_n \left[ \mathcal{E}_{\neq}(\mu_n) -\frac{n+1}{n^2} \sum_{i=1}^{n} \log (1+ |z_i|^2)  \right]\right)
d\ell_{\CC^n}(z_1,\dots,z_n) .\\
\end{multline*}
We reduced the problem to prove the large deviations lower bound for a  Coulomb gas in $\CC$ with potential $V(z)=\log(1+|z|^2)$ and temperature $\beta_n$. We will prove in Proposition \ref{homeo} that $T^*$ is an homeomorphism from $\mathcal{M}_1(\CC)$ to $\mathcal{M}_1(\mathcal{S}^2) \mid \mu(\{N\})=0\}$ so the set $T^{-1*} B(\sigma,\delta)$ is a neighborhood of $T^{-1*} \sigma$.

The proof of the lower bound can be found in \cite[chapter 5 page 220]{hiaipetz}. We give a proof for the sake of completeness.

We can restrict the proof to the case where $\sigma$ is absolutely continuous with respect to the Lebesgue measure and with density bounded from above and below with rectangle support. Indeed, as the function $-\log|z-w| +\frac{1}{2} \log(1+|z|^2) +\frac{1}{2} \log(1+|w|^2)$ is bounded from below, we can assume that $\sigma$ is supported in a rectangle, as we can approximate $\sigma$ with $\frac{1_{[-m,m]^2}}{\sigma([-m,m]^2)}\sigma$. Convolution with a smooth probability density $\phi_{\varepsilon}$ supported in $[-\varepsilon,\varepsilon]^2$ ensure us the existence of a positive density with respect the Lebesgue measure. The functional: $$E: \mu \mapsto \iint  -\log|z-w| +\frac{1}{2} \log(1+|z|^2) +\frac{1}{2} \log(1+|w|^2)d\mu(z)d\mu(w)$$ is lower semi-continuous, hence we have for $\varepsilon$ small enough:
\[
E(\phi_{\varepsilon} \ast \sigma) \geq E(\sigma)
\]
so we can restrict to measures with positive densities. Those two remarks put together allow us to assume that the measure $\sigma$ is supported in a rectangle $[a,b]\times [c,d]$ with a density $h$ with respect to the Lebesgue measure on $\CC$ satisfying for all $z \in [a,b]\times [c,d]$: 
$$\frac{1}{C} \leq h(z) \leq C$$ for some constant $C>0$. For each $n\in \NN$, let $m=\lfloor \sqrt{n} \rfloor$. Let $x_0=a < x_1 < \dots < x_m=b$ such that for each $j \in \{0, \dots, n-1\}$ 
$$ \sigma([x_j,x_{j+1}]\times [c,d])= \frac{1}{m}.$$
\begin{figure}\label{fig2}
\centering
	\includegraphics[scale=1]{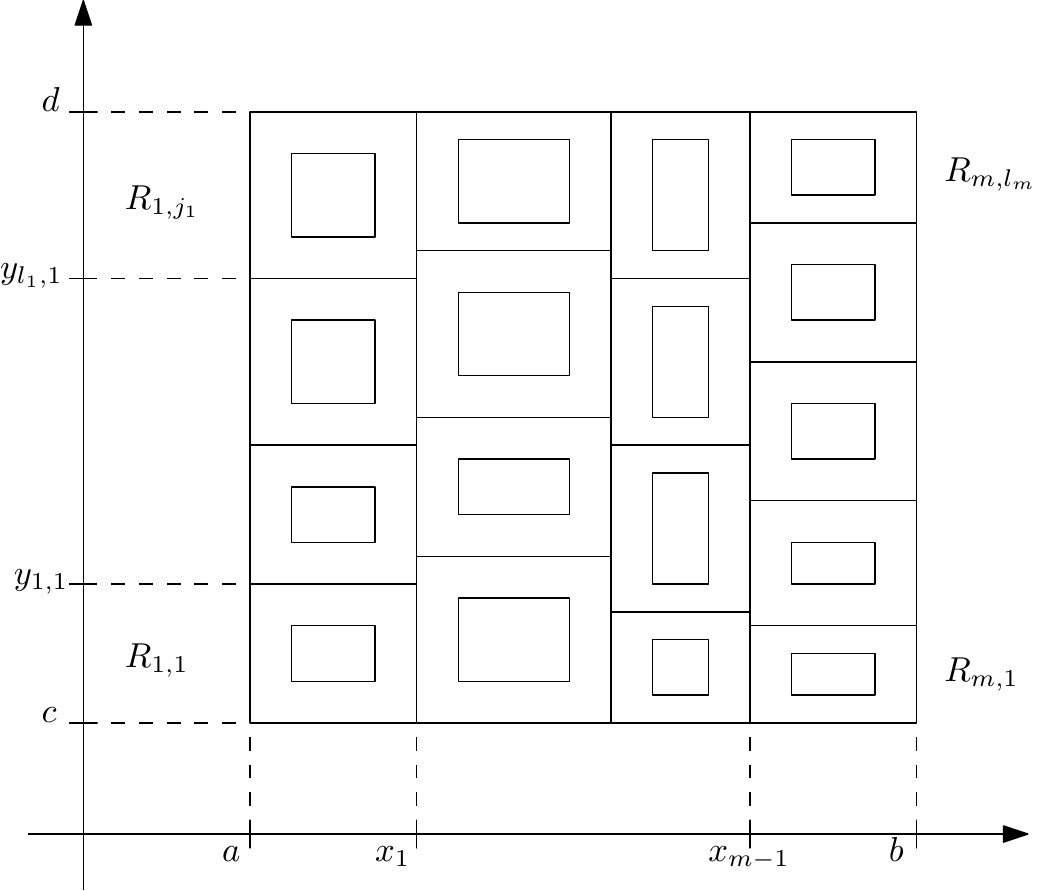}	
	\caption{\label{image} Division or the support in rectangles with $n=17$. So we divide the support in $4$ columns which we divide in $4$ or $5$ rectangles to obtain the total number of $17$ pieces.}
\end{figure}
We have cut the support in vertical slides of equal mass. We divide each slide in rectangles of equal mass (see Figure 2)and we adjust their number in order to have a total of $n$ parts.
As $m^2 \leq n \leq m(m+2)$, we can find $l_1, \dots ,l_m$ satisfying $\sum_{k=1}^{n}l_k=n$ and for each $j\in \{0,n-1\}$ a set of points $y_{j,0}=c < y_{j,1}< \dots <y_{j,l_j}=d$ such that for every $i \in \{0,\dots,n-1\}$ and every $j\in \{0,l_i-1\}$ we have:
$$\sigma([x_{i-1},x_i]\times [y_{j-1},y_{j}])= \frac{1}{ml_i}.$$
This construction gives us a set of rectangles. For each one of them, we consider the smaller rectangle with same center as the original one but homothetic with ratio $\frac{1}{2}$ (see Figure 2 for an illustration). We call them $R_{i,j}=[x_{i-1}',x_{i}']\times [y_{j-1}',y_{j}']$. As the density $h$ is bounded from above and from below we see that the diameter of the rectangles uniformly tends towards 0 as $n$ goes to infinity. More precisely we have:
$$ \lim\limits_{n \rightarrow \infty} \max_{i,j} \mathrm{diam}(R_{i,j}) \rightarrow 0.$$
We will need to control the area of each rectangle, we want this area no to go to zero too fast such that the product of the areas of the rectangles is negligible compared to $e^{- \beta_n}$. One can easily check that for each $i,j$ we have, for some constant $C_1$ depending only of $h$:
$$ \int_{R_{i,j}}d\sigma(z) \geq \frac{C_1}{n} .$$ We label the $n$ rectangles $R_1, \dots, R_n$ (and $R_1', \dots, R_n' $)and we define: $$\Delta_n= \{ (z_1,\dots,z_n)\in \CC^n \mid \forall 1 \leq i \leq n , z_i \in  R_i' \}.$$
Let $f$ be a bounded 1-Lipschitz function. For any $n$  and any $(z_1, \dots , z_n)\in \Delta_n$ we have:
\begin{align*}
\Big\vert \int f(z) d\sigma(z) - \sum_{i,j} \frac{1}{ml_i} f(z_{i,j})\Big\vert & \leq \sum_{i,j} \int_{R_{i,j}} |f(z)-f(z_{i,j})|d\sigma(z)  \\
&  \leq \sum_{i} \int_{R_{i,j}} |z-z_i|d\sigma(z) \\
& \leq \max_{i,j}(\mathrm{diam} R_{i,j}) \xrightarrow[n\rightarrow \infty]{} 0
\end{align*}
and for any bounded function $f$: 
$$ \frac{1}{n} \sum_{i,j}f(z_{i,j}) - \sum_{i,j} \frac{1}{ml_i}f(z_{i,j}) \xrightarrow[n\rightarrow \infty]{} 0.$$
Hence, for any fixed $\delta$, if $n$ is large enough, for any $(z_1,\dots,z_n)\in \Delta_n$: $$\frac{1}{n}\sum_{i=1}^{n} \delta_{z_i} \in B(\sigma, \delta).$$
We define $$\Delta^n= \bigcup\limits_{p\in \mathcal{\sigma}_n} \{(z_1, \dots, z_n)\in \CC^n \mid (z_{p(1)},\dots,z_{p(n)})\in \Delta_n  \}.$$We only made the definition of $\Delta_n$ symmetric, as there was no reason for $z_1$ to be at the top left corner of the support. It is clear that we still have, for any $\delta$ and $n$ large enough, and for any $(z_1,\dots,z_n) \in \Delta^n$, $\frac{1}{n}\sum_{k=1}^{n}\delta_{z_k} \in B(\sigma,\delta)$. Notice that:
$$\text{Vol}(\Delta^n)= n! \text{Vol}(\Delta_n) \geq n!\left(\frac{C_1}{n}\right)^n$$
which implies, as $\beta_n \gg n$ $$\lim_{n\to\infty} \frac{1}{\beta_n} \log \text{Vol}(\Delta^n) = 0.$$
Then we have, for $n$ large enough:
\begin{align*}
	&\int  1_{\mu_n \in T^{-1}B(\sigma, \delta)} \exp \left(-\beta_n \left[\mathcal{E}_{\neq}(\mu_n) +\frac{n+1}{n^2} \sum_{i=1}^{n} \log (1+ |z_i|^2)  \right]\right)
d\ell_{\CC^n}(z_1,\dots,z_n) \\ 
	&\qquad \geq  \int  1_{\Delta^n} \exp\left( -\beta_n \left[\mathcal{E}_{\neq}(\mu_n) +\frac{n+1}{n^2} \sum_{i=1}^{n} \log (1+ |z_i|^2)  \right]\right)
d\ell_{\CC^n}(z_1,\dots,z_n) \\
	 & \qquad\geq  \exp \left( -\beta_n\left[\frac{n+1}{n^2}\sum_{k=1}^{n} \max_{z \in R_i'}\log(1+|z|^2)  - \frac{1}{n^2} \sum_{i \neq j} \min_{z\in R_i' , w \in R_j'} \log|z-w|   \right] \right)\text{Vol}(\Delta^n).
\end{align*}
To obtain the lower bound, it is sufficient to prove that we have:
\begin{equation}
\lim\limits_{n\rightarrow \infty} \sum_{k=1}^{n} \max_{z \in R_i'}\log(1+|z|^2) = \int \log(1+|z|^2) d\sigma(z)
\end{equation} 
and that 
\begin{equation}
	\varliminf_{n \rightarrow \infty} \frac{1}{n^2} \sum_{i \neq j} \min_{z\in R_i' , w \in R_j'} \log|z-w| \geq \iint \log|z-w|d\sigma(z)d\sigma(w)= \mathcal{E}(\sigma).
\end{equation}
The fist limit is easy to prove as $z \mapsto \log(1+|z|^2)$ is uniformly continuous on the support of $\sigma$ (which is a rectangle) and is deduced from the definition of the Riemann integral.

As $d_n=\max_{i,j} \text{diam}(R_{i,j})$ is of order $\frac{1}{\sqrt{n}}$ and $\min_{R_i',R_j'}|z-w|$ is also of order $\frac{1}{\sqrt{n}}$, then, thanks to the bound:
\begin{equation}
\max_{R_i,R_j}|z-w| \leq 2d_n + \min_{R_i',R_j'}|z-w|
\end{equation}
there exists a constant $A>0$ such that 
\begin{equation}
A \min_{z\in R_i' , w \in R_j'} |z-w| \geq  \max_{z\in R_i , w \in R_j} |z-w|.
\end{equation}
This relation is the reason why we reduced the sizes of the rectangles in the construction so that we can control the distance between the rectangles.
To end the proof we notice that, for any $\varepsilon >0$, we have:
\begin{equation}
\lim\limits_{n \rightarrow \infty} \frac{1}{n^2}\# \{ i \neq j \mid \frac{\max_{z\in R_i , w \in R_j} |z-w|}{\min_{z\in R_i' , w \in R_j'} |z-w|}\leq 1+\varepsilon \}=1. 
\end{equation} 
Indeed, for any fixed $\varepsilon>0$, the cardinal of the complement of the set considered above is $O(n)$ as this condition is verified as soon as the rectangles are not too close.

We call
$$B= \mathcal{E}(\sigma) - \sum_{i\neq j} \log \left(\min_{z \in R_i', w \in R_j'}|z-w|\right) \frac{1}{ml_i \times m l_j }.$$
Since
\begin{equation}
\iint \log |z-w|d\sigma(z) d\sigma(w) \leq \sum_{i\neq j} \log \max_{z \in R_i, w \in R_j}|z-w| \frac{1}{ml_i \times m l_j}
\end{equation} 
then for every $i$, $n\leq ml_i$ then, for every $\varepsilon>0$, 
then we have :
\begin{align*}
B   \leq & \sum_{i\neq j} \log \left(\max_{z \in R_i, w \in R_j}|z-w|\right) \frac{1}{ml_i \times m l_j}  - \sum_{i\neq j} \log \left(\min_{z \in R_i', w \in R_j'}|z-w|\right) \frac{1}{ml_i \times m l_j } \\  
 \leq &\sum_{i\neq j} \log \left(\frac{\max_{z\in R_i , w \in R_j} |z-w|}{\min_{z\in R_i' , w \in R_j'} |z-w|}\right) \frac{1}{ml_i \times m l_j} \\
 \leq &\frac{1}{n^2} \sum_{i\neq j} \log \left( \frac{\max_{z\in R_i , w \in R_j} |z-w|}{\min_{z\in R_i' , w \in R_j'} |z-w|}\right) \\
 \leq & \frac{1}{n^2}\# \{ i \neq j \mid \frac{\max_{z\in R_i , w \in R_j} |z-w|}{\min_{z\in R_i' , w \in R_j'} |z-w|}\leq 1+\varepsilon \} \log(1+\varepsilon)  \\  & \qquad \qquad \qquad +  \frac{1}{n^2}\left[1-\# \{ i \neq j \mid \frac{\max_{z\in R_i , w \in R_j} |z-w|}{\min_{z\in R_i' , w \in R_j'} |z-w|}\}\right] \log A.
\end{align*}
Then we take the limit superior in both sides, and the limit when $\varepsilon \rightarrow 0$
\begin{align*}
\mathcal{E}(\sigma) - \varliminf_{n\to\infty} \frac{1}{n^2}\sum_{i\neq j} \log &\left( \min_{z \in R_i', w \in R_j'}|z-w|\right) \frac{1}{ml_i \times m l_j } \\ & \leq \varlimsup_{n\to\infty} \frac{1}{n^2} \sum_{i\neq j} \log \left(\frac{\max_{z\in R_i , w \in R_j} |z-w|}{\min_{z\in R_i' , w \in R_j'} |z-w|}\right)=0
\end{align*}  
which ends the proof of the lower bound in the complex case.

\begin{itemize}
\item \textbf{Large deviations lower bound in the real case}
\end{itemize}

In this section, we prove the lower bound of the large deviation principle of Proposition \ref{thmreel}. We define the non-normalised measures:
$$\PP_{n,k}(x_1,\dots,x_n)=\exp \left( -\beta_n\frac{1}{2} H(x_1,\dots,x_n) \right)dL_{n,k}(x_1,\dots,x_n)$$ and $$\tilde{\PP}= 	\sum_{k=0}^{\lfloor n/2 \rfloor} \PP_{n,k}.$$
These measures are not probability measures, as we omitted the factors $Z_{n,k}$. We will recover the lower bound for the initial measure afterwards, using a uniform control over the constants $Z_{n,k}$.
We prove that  for any $\sigma \in \mathcal{M}_1(\mathcal{S}^2)$:
$$\lim\limits_{\delta \rightarrow O} \varliminf_{n\rightarrow \infty} \frac{1}{\beta_n}\tilde{\PP}(\bar{\mu}_n \in B(\sigma, \delta)) \geq -I_{\mathcal{S}^2}(\sigma).$$

We can assume that the measure $\sigma$ satisfies $	\tilde{I}_{\mathcal{S}^2}(\sigma)< \infty$ as if the rate function is infinite this bound holds clearly. We notice that it is equivalent to have $\mathcal{E}_{\mathcal{S}^2}(\sigma)<\infty$ and the push-forward of $\sigma$ by the inverse stereographic projection is invariant under conjugation.

The strategy adopted here is very similar to what was done in \cite{benarouszeitouni} for the real Ginibre ensemble. The distribution of $(z_1,\dots,z_n)$ is a mixture between several distributions, each distribution being related to the number of real zeros. For any $\varepsilon>0$, if $\delta$ is small enough, using \eqref{ineqbidon} and as $J_{\mathcal{S}^2}$ is continuous:
\begin{align*}
	&\tilde{\PP}(\bar{\mu}_n \in B(\sigma, \delta))  \\ \quad & =    \sum_{k=0}^{\lfloor n/2 \rfloor} \int  \exp \left(-\beta_n\left[  \frac{1}{2}\mathcal{E}_{\neq}(\bar{\mu}_n)  -\frac{n+1}{2n^2} \log \int\prod_{i=1}^N |x-x_i|^2 2^ndT^*\nu_S(x) \right]\right)   1_{\bar{\mu}_n \in B(\sigma, \delta)}d\rho_{n,k} \\
	& \geq \sum_{k=0}^{\lfloor n/2 \rfloor} \int  \exp\left( -\beta_n\left[  \frac{1}{2}\mathcal{E}_{\neq}(\bar{\mu}_n)  -\frac{n+1}{2n} J_{\mathcal{S}^2}(\bar{\mu}_n) \right] \right) 1_{\bar{\mu}_n \in B(\sigma, \delta)}d\rho_{n,k} \\	
	&\geq  \exp \left(- \beta_n \left[ \frac{n+1}{n}(J_{\mathcal{S}^2}(\sigma) + \varepsilon )\right]\right)\sum_{k=0}^{\lfloor n/2 \rfloor} \int   \exp \left( -\beta_n  \frac{1}{2}\mathcal{E}_{\neq}(\bar{\mu}_n) \right)    1_{\bar{\mu}_n \in B(\sigma, \delta)}d\rho_{n,k} \\
	&\geq  \exp \left(- \beta_n \left[ \frac{n+1}{n}(J_{\mathcal{S}^2}(\sigma) + \varepsilon )\right]\right) \int   \exp \left(-\beta_n  \frac{1}{2}\mathcal{E}_{\neq}(\bar{\mu}_n)   \right) 1_{\bar{\mu}_n \in B(\sigma, \delta)}d\rho_{n,\lfloor n/2 \rfloor}.
\end{align*}
As we deal with a lower bound, we can only consider the last term of the sum, corresponding to zero or one real root (if $n$ is even or odd). Like in the complex case, we use the inverse stereographic projection to express the last integral:
\begin{multline*}
	\int  1_{\bar{\mu}_n \in B(\sigma, \delta)}\exp \left(-\beta_n  \frac{1}{2}\mathcal{E}_{\neq}(\bar{\mu}_n)   \right) d\rho_{n,\lfloor n/2 \rfloor}   \\ = \int 1_{\mu_n \in T^{-1}B(\sigma, \delta)} \exp \left(-\beta_n \left[\frac{1}{2}\mathcal{E}_{\neq}(\mu_n) -\frac{n+1}{2n^2} \sum_{i=1}^{n} \log (1+ |z_i|^2)  \right]\right)
	d\ell_{n,\lfloor n/2 \rfloor}(z_1,\dots,z_n).
\end{multline*}
Notice that when $n$ is odd, the first coordinate is always real. Like in the complex case, we reduced the proof of the lower bound to the proof of the lower bound for a Coulomb gas in the plane with potential $\log(1+|z|^2)$, except that this gas must have at most one particle on the real axis (depending on $n$). Strategy of the proof is exactly the same as previously: for regular measures (supported in a rectangle with density with respect to the Lebesgue measure bounded from above and from below), we prove the lower bound by approximating the measure with an atomic measure whose atoms can be anywhere in well chosen rectangles. To perform the approximation, we have to show that we can still make this construction respecting the invariance under conjugation and the fact that we must have at most one real root. 

When $n$ is even, we cut the upper half of the support in boxes of mass $\frac{1}{2\lfloor \sqrt{n}\rfloor}$ and take the half-sized sub-boxes (the $R_i$'s and the $R_i'$'s), illustrated in Figure 3. Then we describe the lower part of the support by considering the conjugates of the rectangles $\bar{R}_i$ and $\bar{R}_i'$. We define $$\Delta_n= \{ (z_1,\dots,z_n)\in \CC^n \mid \forall 1 \leq i \leq n , z_i \in  R_i' \}$$ and the rest of the construction is exactly the same.

\begin{figure}[h!]
\centering
	\includegraphics[scale=1]{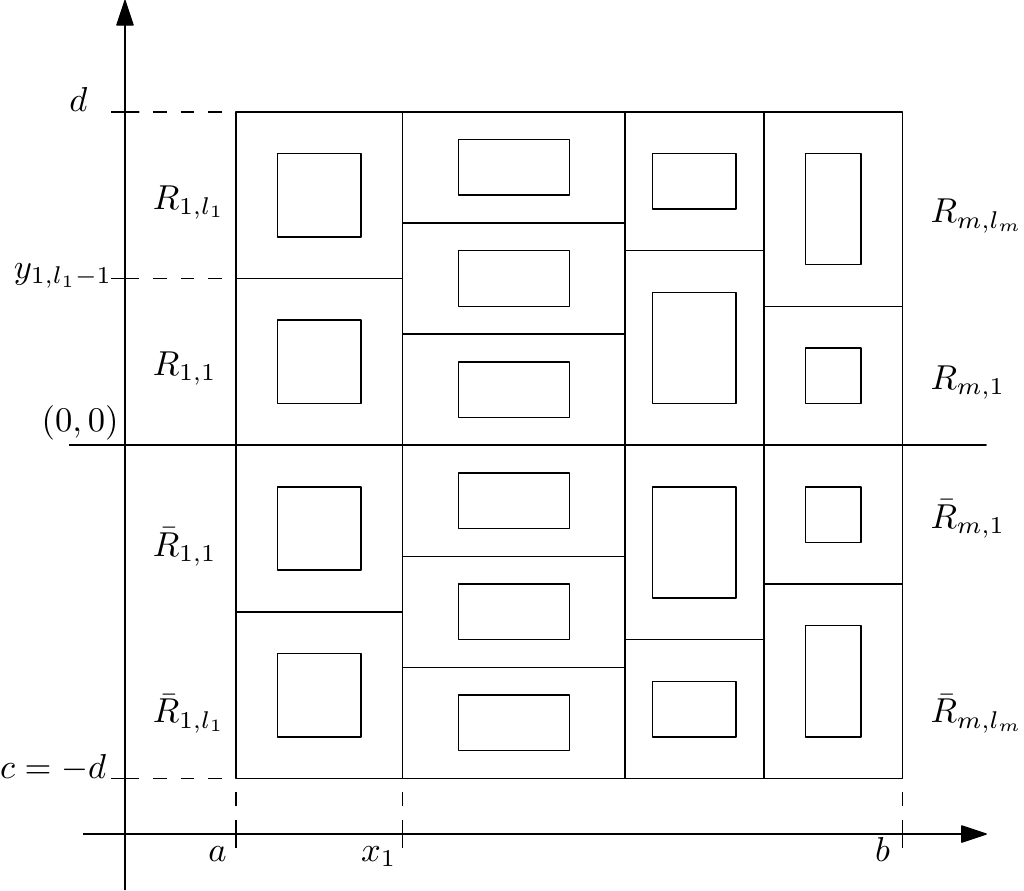}
		\caption{\label{image} Division or the support in rectangles, $n$ even. Here $n=18$, we divide the support in $4$ columns. Then we divide the columns in order to obtain a total number of $18$ rectangles.}
	
\end{figure}
When $n$ is odd, the construction is nearly the same, except that we must place one point on the real line. We divide the upper half of the support into columns of equal mass $\frac{1}{2\lfloor\sqrt{ n }\rfloor}$, then we divide each column in rectangles of equal mass such that the total number of rectangles is $\frac{n+1}{2}$. We consider the conjugates of the rectangles  so that we have cut the support of $\sigma$ in rectangles of known mass respecting the symmetry by conjugation (Figure 4). Finally, on the 1st column, we consider the union of the two rectangles touching the real axis. Hence, we have a partition of the support in $n$ rectangles. Now we can consider the smaller rectangles like we used to do, except for the one crossing the real axis for which we only keep the intersection of the smaller rectangle with $\RR$, represented by a thicker line in Figure 4.

We can also define  $\Delta_n= \{ (z_1,\dots,z_n)\in \RR \times\CC^{n-1} \mid \forall 1 \leq i \leq n , z_i \in  R_i' \}$
which differs from the previous construction only by the fact that the first variable is real, and $$\Delta^n= \{ (z_1,\dots,z_n)\in \RR \times\CC^{n-1} \mid \exists p \in \mathcal{\sigma}\{2,\dots,n\} , (z_1,z_{p(2)},\dots,z_{p(n)})\in \Delta_n \} .$$
\begin{figure}[h!]
	\centering
	\includegraphics[scale=1.3]{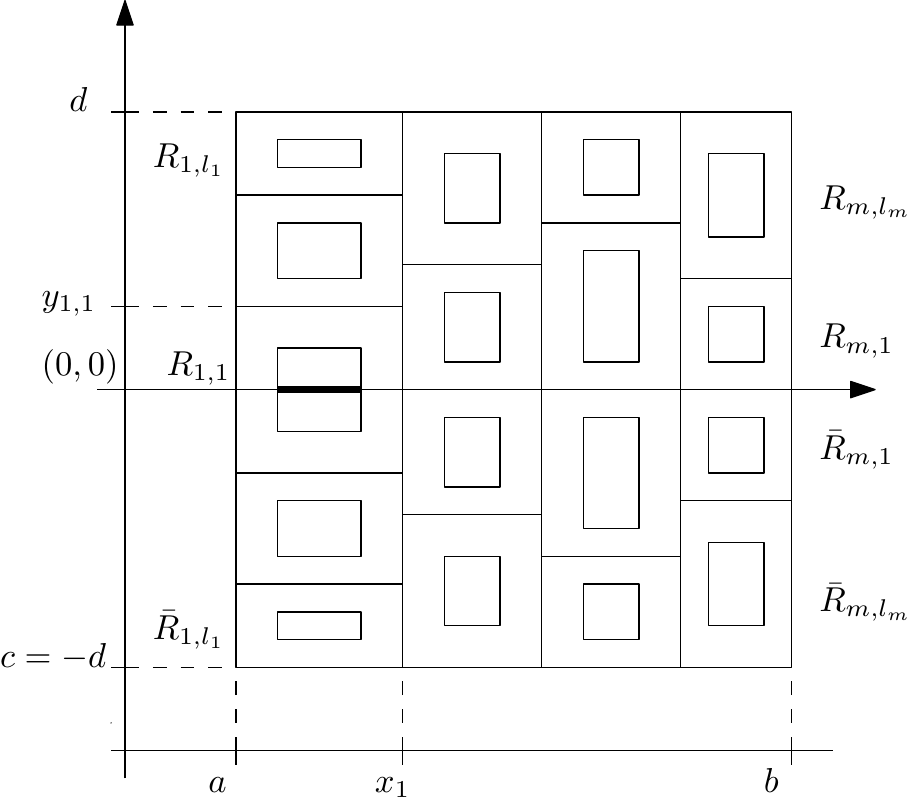}
	\caption{\label{image} Division of the support in rectangles, $n$ odd. Here $n=17$ we divide the support in 4 columns. The thick line corresponds to the first rectangle, which is flat.}
\end{figure}
The rest of the inequalities are valid for this construction and it allows us to end the proof of the lower bound in the real case.

\subsubsection{Large deviation principle for normalized measures.} \label{getridofZ}
In this subsection we obtain the full large deviation principle on the sphere by treating the normalizing constants. This technique is classical in large deviations for Coulomb gases. Unfortunately, the real case is less direct than the complex case and requires some control over the $Z_{n,k}$ constants.
Here we have an explicit formula for the constants $Z_n$ and $Z_{n,k}$ in Theorem \ref{the:rootsexpo}. 

As $\lim_{n\to\infty}\frac{1}{\beta_n} \log Z_n = 0$, the large deviation principle holds for the normalized measures.
For any open $O$ set in $\mathcal{M}_1(\mathcal{S}^2)$, we have:
\begin{equation}
-\inf_{O} I_{\mathcal{S}^2} \leq \varliminf_{n\to\infty} \frac{1}{\beta_n} \log Z_n \PP(\bar{\mu}_n \in O) \leq \varlimsup_{n\to\infty}  \frac{1}{\beta_n} \log Z_n \PP(\bar{\mu}_n \in O) \leq -\inf_{ \mathrm{clo}(O)} I_{\mathcal{S}^2}.
\end{equation} 
By taking $O=\mathcal{M}_1(\mathcal{S}^2)$ we get: 
\begin{equation}\label{trick}
\lim_{n \rightarrow \infty }\frac{1}{\beta_n} \log Z_n = -\inf_{\nu \in \mathcal{M}_1(\mathcal{S}^2)} I_{\mathcal{S}^2}(\nu)=0.
\end{equation}
In this case, we know that the rate function $I$ reaches its minimum for the uniform measure on the equator as we know that the random sequence $\mu_n$ converges almost surely weakly towards the uniform measure on the unit circle. As $I_{\mathcal{S}^2}(T^*\nu_S)=I(\nu_S)=0$ it will imply the (already known) almost sure weak convergence of $(\mu_n)_{n \in \NN}$ towards $\nu_S$. This ends the proof of Proposition \ref{thmcomplexe}.

In the real case, we have to check that 
\begin{equation} \label{assumption0}
\lim_{n \rightarrow \infty}\min_k Z_{n,k}=\lim_{n \rightarrow \infty}\max_k Z_{n,k} = - \inf \tilde{I}_{\mathcal{S}^2}=0.
\end{equation} This is true when $\beta_n=n^2$ (see Proposition \ref{limits}). For general $\beta_n$, we have to assume that \eqref{assumption0} is true.
Then for any set $A\in \mathcal{M}_1(\mathcal{S}2)$ we have:
\begin{equation}\label{trick2}
\varlimsup_{n\to\infty} \frac{1}{\beta_n} \log \sum_{k=0}^{\lfloor n/2 \rfloor} \frac{1}{Z_{n,k}} \PP_{n,k}(A) \leq \varlimsup_{n\to\infty} -\frac{1}{\beta_n}\log \min_k Z_{n,k} - \inf_{\mathrm{clo}A} I_{\mathcal{S}^2}
\end{equation}
and 
\begin{equation}
\varliminf_{n\to\infty}\frac{1}{\beta_n} \log \sum_{k=0}^{\lfloor n/2 \rfloor} \frac{1}{Z_{n,k}} \PP_{n,k}(A) \geq \varliminf_{n\to\infty} -\frac{1}{\beta_n} \log \max _k Z_{n,k} - \inf_{\mathrm{int}(A)}I_{\mathcal{S}^2}
\end{equation}
which proves the full large deviation principle in $\mathcal{M}_1(\mathcal{S}^2)$.

\subsection{Step 4: Going back on the plane}
We have proved large deviation principles for the real and complex case on the sphere. The next proposition is taken from \cite[Lemma 2.1]{hardy}. We recall that the point $N=(0,0,1)$ is the north pole of the sphere.

\begin{proposition}[Correspondence between $\CC$ and $\mathcal{S}^2 \setminus \{N\}$]\label{homeo}
	$T^*$ is an homeomorphism from $\mathcal{M}_1(\CC)$ to $\{ \mu \in \mathcal{M}_1(\mathcal{S}^2) \mid \mu(\{N\})=0\}$.
\end{proposition}

\begin{proof}
As $T$ is a continuous function, $\mu \mapsto T^*\mu$ is  continuous for the weak topology. As $T$ is a bijection from $\CC$ to $\mathcal{S}^2 \setminus \{N\}$, it follows that $T^*$ is a bijection with inverse $(T^{-1})^*$. We only have to prove the continuity of $(T^{-1})^*$. Let $(\nu_n)_{n\in \NN}$ be a sequence of measures in $\{ \nu \in \mathcal{M}_1(\mathcal{S}^2) \mid \nu(\{N\})=0\}$ that converges in $\{ \nu \in \mathcal{M}_1(\mathcal{S}^2) \mid \nu(\{N\})=0\}$. Let $\nu_{\infty}$ the limit of this sequence. By outer regularity of $\nu_{\infty}$ and the Portmanteau theorem, for any $\varepsilon$, there is an open set $B$ such that:
$$ \varlimsup_{n\to\infty} \mu_n (B) \leq \mu (B) \leq \varepsilon. $$
The last inequality shows that the sequence $((T^{-1})^* \nu_n)_n$ is tight. It is easy to see that $((T^{-1})^* \mu_n)_n$ converges vaguely towards $(T^{-1})^* \nu_{\infty}$ hence weakly.
\end{proof}

\begin{proposition}[Rate functions] \label{equivalence}
	For any measure $\mu \in \mathcal{M}_1(\CC)$ , $I_{\mathcal{S}^2}(T^*\mu)= I(\mu)$.
\end{proposition}

\begin{proof}[Proof of Proposition \ref{equivalence}]
\begin{multline*}
-\iint \log|z-w|d\mu(z)d\mu(w) + 2 \sup_{z\in S^1} \int \log |z-w|d\mu(w)  \\  
\qquad= -\iint \left(\log|T(z)-T(w)|-\frac{1}{2}(\log(1-|T(z)|^2)+ \log(1-|T(w)|^2) \right) d\mu(z)d\mu(w) \\
\qquad + 2 \sup_{z\in S^1} \int \left( \log|T(z)-T(w)| -\frac{1}{2}(\log(1-|T(z)|^2)+ \log(1-|T(w)|^2) \right) d\mu(w).
\end{multline*}
Hence we obtain:
\begin{align*}
I(\mu)  & = -\iint \log|T(z)-T(w)|d\mu(z)d\mu(w) +2 \sup_{z\in S}\int \log |T(z)-T(w)|d\mu(w) +\log 2 \\
& = I_{\mathcal{S}^2}(T^*\mu).
\end{align*}
\end{proof}
From the large deviation principles proved on the sphere, we can now deduce  the large deviation principles on the plane for the complex case.

\begin{proof}[Proof of Theorem \ref{LDPKac}]
	Thanks to the inclusion principle \cite[lemma 4.1.5]{dembozeitouni}, the random sequence $(T^*\mu_n)_{n\in \NN}$ satisfies a large deviation principle in $$\{ \mu \in \mathcal{M}_1(\mathcal{S}^2) \mid \mu(\{N\})=0\}$$ with speed $\beta_n$ and good rate function $I_{\mathcal{S}^2}$. Then by  the contraction principle \cite[Theorem 4.2.1]{dembozeitouni} along $(T^{-1})^*$, the sequence $(\mu_n)_{n \in \NN}$ satisfies a large deviation principle with the same speed and good rate function $I$ thanks to Proposition \ref{equivalence}. The function $I$ is a good rate function as we have already proved that $I_{\mathcal{S}^2}$ is a good rate function.
\end{proof}

\begin{proof}[Proof of Theorem \ref{LDPKacreel}]
	In the real case, the proof is exactly the same. We use the inclusion principle and the contraction principle to obtain a large deviation principle with speed $\beta_n$ and good rate function $\tilde{I}$ (using Proposition \ref{equivalence} again).
\end{proof}

\section{Large deviations for Elliptic polynomials}

In the last section we saw a large deviation principle for the empirical measures of zeros of random Kac polynomials. In this section, we study the gases \eqref{ellipticcomplexe} and \eqref{elliptiquereel}. We prove Theorems \ref{LDPelliptique} and \ref{LDPelliptiquereel} following the same steps as previously.

\subsection{Step 1: Distribution of the roots.}
\begin{theorem}[Distribution of the roots of elliptic polynomials] \label{loielliptique}
	The family of polynomials $\sqrt{n+1}\binom{n}{k}^{1/2}X^k$ are an orthonormal basis in $\CC_n[X]$ for the scalar product:
	$$ \langle P,Q \rangle = \int P(z)\overline{Q(z)} \frac{1}{ (1+|z|^2)^{n}} \frac{d\ell_{\CC}(z)}{\pi(1+|z|^2)^2}.$$
In the complex case the distribution of the roots of $ P_n= \sum_{k=1}^{n} \binom{n}{k}^{1/2} a_k X^k$ is given by:
$$\frac{1}{Z_n}	\frac{\prod_{i<j} |z_i-z_j|^2}{\left(\int \prod_{k=1}^{n}|z-z_i|^2\frac{1}{ (1+|z|^2)^{n}} \frac{d\ell_{\CC}(z)}{\pi(1+|z|^2)^2} \right)^{n+1}}d\ell_{\CC^n}(z_1,\dots,z_n)$$
where $$Z_n = \frac{\pi^{n-1}}{n! |A_n|^2}$$
is a normalizing constant and where $|A_n|^2$ is the Jacobian of the change of variables from the canonical basis of $\CC_n[X]$ to the orthonormal basis $(R_0,\dots,R_n)$. This distribution can be written:
$$\frac{1}{Z_n}\exp \left( -\beta_n H_E(z_1,\dots,z_n)\right)d\ell_{\CC^n}(z_1,\dots,z_n).$$
In the real case, the distribution of the roots is given by:
\begin{equation*}
\sum_{k=0}^{\lfloor n/2 \rfloor} \frac{1}{Z_{n,k}} \frac{\prod_{i<j} |z_i-z_j|}{\left(\int \prod_{k=1}^{n}|z-z_i|^2\frac{1}{ (1+|z|^2)^{n}} \frac{d\ell_{\CC}(z)}{\pi(1+|z|^2)^2} \right)^{(n+1)/2}} d\ell_{n,k}(z_1,\dots,z_n)
\end{equation*}
where \begin{equation*} Z_{n,k}=\frac{k!(n-2k)!\pi^{\frac{n+1}{2}}}{2^k \Gamma(\frac{n+1}{2}) |A_n|}
\end{equation*}
and $|A_n|$ is the Jacobian of the changes of variables from the canonical basis of $\RR_n[X]$ to the orthonormal basis $(R_0, \dots, R_n)$. This distribution can also be written
\begin{equation*}
\sum_{k=0}^{\lfloor n/2 \rfloor} \frac{1}{Z_{n,k}}  \exp\left( -\beta_n\frac{1}{2}H_E(z_1,\dots,z_n)\right)  d\ell_{n,k}(z_1,\dots,z_n).
\end{equation*}
\end{theorem}
These polynomials are handled by the article of Zeitouni and Zelditch \cite{zeitounizelditch} with the complex Cauchy (Fubini-Study) measure and the weight $\phi(z)= \log(1+|z|^2)$.

\begin{proof}[Proof of Theorem \ref{loielliptique}]
First we prove that the polynomials $\sqrt{n+1}\binom{n}{k}^{1/2}X^k$ are an orthonormal basis in $\CC_n[X]$. As the weight and the measure are radial, it is clear that this family is orthogonal. We only have to compute the norm of each polynomial:
\begin{align*}
\int_{\CC} \frac{|z|^{2k}}{(1+|z|^2)^{n+2}}d\ell_{\CC}(z) =&  2\pi \int_{\RR^+} \frac{r^{2k+1}}{(1+r^2)^{n+2}}dr \\
= & \pi \int_{\RR^+} \frac{u^{k}}{(1+u)^{n+2}}du \\
= & \pi \frac{k}{n+1} \int_{\RR^+} \frac{u^{k-1}}{(1+u)^{n+1}} \\
= & \pi (n+1) \binom{n}{k}^{-1}.
\end{align*}
The computation of the distribution of the roots of random elliptic polynomials with complex coefficients is a change of variables. 

Let $(z_1,\dots,z_n)$ be the zeros of $P_n= \sum_{k=0}^{n} a_k \binom{n}{k}^{1/2}X^k$, then we consider $$G(z_1,\dots,z_n,a_n)= (a_0, \dots,a_n).$$
To compute the Jacobian determinant of $G$, we use the following decomposition:

\[\xymatrix{ (z_1,\dots,z_n,a_n) \ar[rr]^{G} \ar[rd]_U && (a_0,\dots,a_n)  \\ & (b_0,\dots,b_n) \ar[ru]_V }\]
where $U$ is the function giving the coefficients in the canonical basis of $P_n$ from its roots and leading coefficient and $V$ is the change of basis from the canonical basis to the basis $\binom{n}{k}^{1/2}X^k$. We have already seen that $\text{Jac}(U)=|a_n|^{2n} \prod_{i<j}|z_i-z_j|^2$. We could compute the Jacobian determinant of $V$, but we will just call this quantity $|A_n|^2$.Hence, the real Jacobian determinant of $G$ is:
\begin{equation*}
	|\mathrm{Jac}(G)|^2= |A_n|^2 |a_n|^{2n} \prod_{i<j}|z_i-z_j|^2.
\end{equation*}

The end of the proof is the same as for Kac polynomials, the density of the random vector $(a_0, \dots,a_n)$ being $\frac{e^{-\|P\|^2}}{\pi^{n+1}}$, we only have to integrate the distribution of $(z_1,\dots,z_n,a_n)$ to obtain the announced distribution for the complex case.

In the real case, we use Zaporozhets' computation \cite{zaporozhets} to obtain the distribution of $(z_0,\dots,z_n,a_n)$ in function of the distribution of the coefficients in the canonical basis $(b_0,\dots,b_n)$ and we use the additional change of variables from the canonical basis to the basis $\sqrt{n+1}\binom{n}{k}^{1/2}X^k$. The real Jacobian determinant of this change of variables is $|A_n|$, so we obtain the distribution of $(z_1,\dots,z_n,a_n)$:
\begin{align*}
\sum_{k=0}^{\lfloor n/2 \rfloor} & \frac{|A_n|2^k}{k!(n-2k)!\pi^{\frac{n+1}{2}}}  |a_n|^n \prod_{i<j}|z_i-z_j| \\ & \qquad \times  \exp \left(-\int |a_n|^2 \frac{\prod_{j=1}^{n}|z-z_j|^2}{(1+|z|^2)^n}\frac{d\ell_{\CC}(z)}{\pi (1+|z|^2)^2}\right) d\ell_{n,k}(z_1,\dots,z_n) d\ell_{\CC}(a_n).
\end{align*}
and we integrate with respect the variable $a_n$.

\end{proof}

\subsection{Step 2: Large deviations in \texorpdfstring{$\mathcal{M}_1(\mathcal{S}^2)$}{M1(S2)}}
\begin{proposition}[Pushing elliptic polynomials on the sphere]
	Let $(z_1,\dots,z_n)$ be the zeros of $P_n$ in the complex case, then the law of $(T(z_1),\dots,T(z_n))$ is absolutely continuous with respect to the push forward by $T$ \eqref{stereo} of the Lebesgue measure on $\CC$ with density:
	
	$$\frac{\prod_{i<j} |x_i-x_j|^2}{(\int \prod_{j=1}^n |x-x_j|^2 d\nu_{\mathcal{S}^2}(x))^{n+1}}  \times \prod_{i=1}^{n} (1-|x_i|^2)^2$$
where $\nu_{\mathcal{S}^2}$ is the uniform measure on $\mathcal{S}^2$. Recall that $\kappa_n$ is defined in Proposition \ref{pushcomplexkac}.
This law can be written in the form:
	$$\frac{1}{Z_n}\exp \left(-\beta_n \left[  -\frac{1}{n^2}\sum_{i\neq j} \log|x_i-x_j| + \frac{n+1}{n^2}\log \int \prod_{j=1}^n |x-x_j|^2 d\nu_{\mathcal{S}^2}(x)\right]\right) d\kappa_n. $$

\end{proposition}

\begin{proof}
The proof is nearly the same as the proof of Proposition \ref{pushcomplexkac}. We use the relations \eqref{relations} and \eqref{relation2}:
 \begin{equation*}
\prod_{i<j}|z_i-z_j|^2=\prod_{i<j} \frac{|T(z_i)-T(z_j)|^2}{(1-|T(z_i)|^2)(1-|T(z_j)|^2)}  
 \end{equation*}
 and
 \begin{equation*}
\int \prod_{k=1}^{n}|z-z_i|^2\frac{1}{ (1+|z|^2)^{n}} \frac{d\ell_{\CC}(z)}{\pi(1+|z|^2)^2} =\frac{\displaystyle\int\prod_{i=1}^n |T(z)-T(z_i)|^2 \frac{(1-|T(z)|^2)^2d\ell_{\CC}(z)}{\pi} }{\prod_{i=1}^{n}(1-|T(z_i)|^2)}
 \end{equation*}
 to obtain the density:
 \begin{equation*}
\frac{\prod_{i<j} |T(z_i)-T(z_j)|^2}{(\int\prod_{i=1}^n |T(z)-T(z_i)|^2  d\nu_{\mathcal{S}^2})^{n+1}} \prod_{i=1}^{n} (1-|T(z_i)|^2)^2.
 \end{equation*}
\end{proof}

\begin{proposition}[Pushing the real case on the sphere]\label{pushellipticcomplex}
Let $(z_1,\dots,z_n)$ be the zeros of $P_n$ in the real case, then the law of $(T(z_1),\dots,T(z_n))$ is:
$$\sum_{k=0}^{\lfloor n/2 \rfloor} \frac{1}{Z_{n,k}} \frac{ \prod_{i<j}|x_i-x_j|\times \prod_{i=1}^{n} (1-|x_i|^2)}{ (\int\prod_{i=1}^N |x-x_i|^2 d\nu_{\mathcal{S}^2}(x))^{(n+1)/2}}  dL_{n,k}(x_1,\dots,x_n).$$
Where $\rho_{n,k}$ is like in Proposition \ref{pushreelkac}.
This law can be written:
\begin{multline*}
\sum_{k=0}^{\lfloor n/2 \rfloor} \frac{1}{Z_{n,k}}  \exp \left( -\beta_n\left[  \frac{1}{2}	\mathcal{E}_{\neq}(\mu_n)  -\frac{n+1}{2n^2} \log \int\prod_{i=1}^N |x-x_i|^2 d\nu_{\mathcal{S}^2}(x) \right]\right) d\rho_{n,k}.
\end{multline*}
	
\end{proposition}
The proof of this proposition is the same as Proposition \ref{pushellipticcomplex}.
We can now state theorem on the sphere.

\begin{proposition}[Large deviation principle in $\mathcal{M}_1(\mathcal{S}^2)$]\label{LDPelliptiquesphere}
	In the complex case, the sequence of empirical measures satisfy a large deviation principle with speed $n^2$ in $\mathcal{M}_1(\mathcal{S}^2)$ with good rate function $I_{E,\mathcal{S}^2}-\inf I_{E,\mathcal{S}^2}$ where $$I_{E,\mathcal{S}^2}(\mu)= -\iint \log|x-y|d\mu(x)d\mu(y) + \sup_{x \in \mathcal{S}^2} \int \log|x-y|^2 d\mu(y) .$$
	In the real case, the sequence of empirical measures also satisfies a large deviation principle with speed $n^2$ and good rate function:
$$
	\tilde{I}_{E,\mathcal{S}^2}(\mu)= \begin{cases}
	 \frac{1}{2}(I_{E,\mathcal{S}^2}(\mu)-\inf I_{E,\mathcal{S}^2}) &\text{if $\mu$ is invariant under the map }z \mapsto \bar{z} \\
  \infty &\text{otherwise}.
\end{cases}
	$$
\end{proposition}

\begin{definition}
We define $J_E: \mathcal{M}_1(\CC) \rightarrow \RR $ by:
\begin{equation}
J_E(\mu)= \sup_{z\in \CC} \{ \int \log |z-w|^2 d\mu(w) - \log(1+|z|^2) \}
\end{equation}
and $J_{E,\mathcal{S}^2}: \mathcal{M}_1(\mathcal{S}^2) \rightarrow \RR$ by:
\begin{equation}
J_{E,\mathcal{S}^2}(\mu)= \sup_{x \in \mathcal{S}^2} \int \log|x-y|^2 d\mu(y).
\end{equation}
\end{definition}

\subsection{Step 3: Proof of the large deviation principles}

\begin{proposition}[Rate function $I_{E,\mathcal{S}^2}$] \label{ratefunctionE}
	\item
	1) The function $J_{E,\mathcal{S}^2}$ is a continuous function for the weak topology of $\mathcal{M}_1(\mathcal{S}^2)$ and is bounded.
	\item
	2) The function $I_{E,\mathcal{S}^2}$ is well defined on $\mathcal{M}_1(\mathcal{S}^2)$, takes its values in $[0,\infty]$ and is finite as soon as the logarithmic energy is finite.
	\item
	3) $I_{E,\mathcal{S}^2}$ is lower semi-continuous.
	\item
	4) $I_{E,\mathcal{S}^2}$ is strictly convex.
\end{proposition}

\begin{proof}[Proof of Proposition \ref{ratefunctionE}]
The proof is exactly the same as the proof of Proposition \ref{ratefunction}. We only have to check that $\mathcal{S}^2$ is a compact set in  $\mathcal{S}^2$, which is non-thin at all his points, which is true.
\end{proof}

\subsubsection{Large deviations upper bound.}
The only thing we need to import the proof of the large deviation principle for non-normalized measures in the Kac case is the Bernstein-Markov inequality that was crucial to prove the upper bound.

\begin{lemma}[Bernstein-Markov for elliptic polynomials]\label{BMelliptique}
	Let $n\in \NN$, then for all $P \in \CC_n[X]$ we have:
	$$ \sup_{\CC}\frac{|P(z)|^2}{(1+|z|^2)^{n}} \leq (n+1) \|P\|^2_{L^2}$$ where $\|P\|_{L^2}^2 = \int |P(z)|^2\frac{1}{(1+|z|^2)^{n}}\frac{d\ell_{\CC}(z)}{\pi(1+|z|^2)^2}$.
\end{lemma}
\begin{proof}[proof of Lemma \ref{BMelliptique}]
	Let $K_n(z,w)= \sum_{i=0}^{n} (n+1) \binom{n}{k} z^k \bar{w}^k$. Then we have:
\begin{equation*}
	 \forall P \in \CC_n[X], P(z)=\int P(w) K(z,w) \frac{1}{(1+|w|^2)^{n}}\frac{dw}{\pi(1+|w|^2)^2}=\langle P,K(. ,w)\rangle
\end{equation*}
	Then by the Cauchy-Schwarz inequality we get, for all $z \in \CC$:
\begin{equation*}
|P(z)|^2 \leq \|P\|_{L^2} \int |K_n(z,w)|^2 \frac{1}{(1+|w|^2)^{n}}\frac{dw}{\pi(1+|w|^2)^2}= \|P\|_{L^2} (n+1) K_n(z,z).
\end{equation*}
	Considering that $\|K_n(z,)\|^2_{L^2}=(n+1)K_n(z,z)= (n+1)(1+|z|^2)^n$ we get:
\begin{equation*}
\sup_{\CC} \frac{|P(z)|^2}{(1+|z|^2)^n}\leq (n+1) \|P\|^2_{L^2}.
\end{equation*}
\end{proof}
From this Bernstein-Markov inequality, we deduce the analogue for this model of Lemma \ref{bernmarko} and Lemma \ref{bernmarko2} and we can easily prove the upper bound for non-normalized measures for elliptic polynomials in both real and complex cases.

\subsubsection{Large deviations lower bound}

As we know that $J_{E,\mathcal{S}^2}$ is a continuous function in $\mathcal{M}_1(\mathcal{S}^2)$, the proof of the lower bound is exactly the same. The same inequalities hold and we can reduce the problem to the classical lower bound for a Coulomb gas with confining potential $\log(1+|z|^2)$.

\subsubsection{Large deviation principles for normalized measures.}
In the complex case, we use the same trick of using the inequalities for the whole space (see \eqref{trick}) to obtain:
\begin{equation*}
\varliminf_{n\to\infty} \frac{1}{\beta_n} \log Z_n = \varlimsup_{n\to\infty} \frac{1}{\beta_n} \log Z_n  = - \inf I_{E,\mathcal{S}^2}
\end{equation*}
so we obtain the full large deviation principle for normalized measures in $\mathcal{M}_1(\mathcal{S}^2)$.
Due to the definition of $Z_n$ given in Theorem \ref{loielliptique}, we have:
\begin{equation} \label{limdetE}
\varliminf_{n\to\infty} \frac{1}{\beta_n} \log |A_n|^2 = \varlimsup_{n\to\infty} \frac{1}{\beta_n} \log |A_n|^2 = - \inf I_{E,\mathcal{S}^2}.
\end{equation}
In the real case, we need a uniform estimate of the $Z_{n,k}$. Thanks to the formula given in Theorem \ref{loielliptique} and the limits given in Proposition \ref{limits}, we notice that equation \eqref{limdetE} implies that:
\begin{equation}\label{assumption}
\lim_{n\to\infty} \frac{1}{\beta_n} \min Z_{n,k} = \lim_{n\to\infty} \frac{1}{\beta_n} \max  Z_{n,k} = \lim_{n\to \infty} \frac{1}{\beta_n}\log |A_n| = -\frac{1}{2} \inf I_{E,\mathcal{S}^2}
\end{equation}
and this allows us to prove the large deviation principle in the real case for normalized measures like in \eqref{trick2}. For general $\beta_n$ we need to assume that \eqref{assumption} is true as we canot obtain a uniform controlin $k$ of the constants $Z_{n,k}$. To control the constants $Z_n,k$ we use \eqref{limdetE} which comes from the analysis of the complex case. This proof relies the explicit formulae for the constants $Z_{n,k}$ which are not available in general.

\subsection{Step 4: Going back on the plane.}
The only thing to check is that the rate function given by the contraction principle is the rate function that was announced in the theorem. Using the relations \eqref{relations} and \eqref{relation2} in the definition of the rate function $I_{E,\mathcal{S}^2}$ easily ends the proof.

We end the proof of the large deviations principles as in Section \ref{getridofZ}, using the uniform estimates on the $Z_{n,k}$.

\section{General result of Zeitouni and Zelditch}
In this section, we give the general statement of the result obtained by Zeitouni and Zelditch in \cite{zeitounizelditch} and we extend it to the case of real coefficients. We deal with the gases \eqref{gazgeneral} and \eqref{gazgeneral2} associated to the orthogonal polynomials \eqref{polyortho}.
\subsection{Step 1: Distribution of the roots}

\begin{theorem}[Distribution of the roots of $P_n$]\label{generalroots}
In the complex case, the distribution of the random vector $(z_1,\dots,z_n)$ is:
\begin{equation*}
\frac{1}{Z_n}\exp \left(-\beta_n \left[ \mathcal{E}_{\neq}(\mu_n) + \frac{n+1}{n^2}\log \int \prod_{i=1}^{n}|z-z_i|^2 e^{-n\phi(z)}d\nu(z) \right]\right)d\ell_{\CC^n}(z_1,\dots,z_n)
\end{equation*}
where $$Z_n = \frac{\pi^{n-1}}{n! |A_n|^2}$$ and $|A_n|^2$ is the Jacobian of the change of variables from the canonical basis of $\CC_n[X]$ and the orthonormal basis $(R_O,\dots,R_n)$.

When the polynomials $R_0, \dots, R_n$ have real coefficients, in the real case, the distribution of $(z_1,\dots,z_n)$ is given by:
\begin{equation*}
\sum_{k=0}^{\lfloor n/2 \rfloor} \frac{1}{Z_{n,k}}  \exp\left( -\beta_n\frac{1}{2}H_O(z_1,\dots,z_n)\right) \\   d\ell_{n,k}(z_1,\dots,z_n)
\end{equation*}
where
\begin{equation*} Z_{n,k}=\frac{k!(n-2k)!\pi^{\frac{n+1}{2}}}{2^k \Gamma(\frac{n+1}{2}) |A_n|}.
\end{equation*}
\end{theorem}

\begin{proof}[Proof of Theorem \ref{generalroots}]
The proof is the same as the proof of Theorem \ref{loielliptique}. We consider $G(z_1,\dots,z_n,a_n)= (a_0,\dots,a_n)$ where the $a_i$ are the coefficients in the orthonormal basis. 
Then we use the same decomposition:
$$\xymatrix{ (z_1,\dots,z_n,a_n) \ar[rr]^{G} \ar[rd]_U && (a_0,\dots,a_n)  \\ & (b_0,\dots,b_n) \ar[ru]_V }$$
and the same calculation holds to obtain: $$|\mathrm{Jac}(G)|^2= |A_n|^2 |a_n|^{2n} \prod_{i<j}|z_i-z_j|^2$$
where $|A_n|^2$ is the real Jacobian determinant of the change of basis of $\CC_n[X]$.
In the real case, when the $R_k$'s are real polynomials, we can also do the same calculations, using \cite{zaporozhets}.
\end{proof}

\subsection{Step 2: Large deviations on the sphere}

\begin{proposition}[Pushing orthogonal polynomials on the sphere]
	Let $(z_1,\dots,z_n)$ be the zeros of $P_n$ in the complex case, then the law of $(T(z_1),\dots,T(z_n))$ is absolutely continuous with respect to the push forward by $T$ \eqref{stereo} of the Lebesgue measure on $\CC$ with density:
	
	$$\frac{\prod_{i<j} |x_i-x_j|^2}{(\int \prod_{j=1}^n |x-x_j|^2 e^{-n\tilde{\phi}(x)} dT^*\nu(x))^{n+1}}  \times \prod_{i=1}^{n} (1-|x_i|^2)^2.$$
where $\tilde{\phi}(x)= \phi(T^{-1}(x))  +\log(1-|x|^2)$. If $\kappa_n$ is defined as in Proposition \ref{pushcomplexkac} then we can write this law in the form:
	$$\frac{1}{Z_n}\exp \left(-\beta_n \left[  \mathcal{E}_{\neq}(\bar{\mu}_n) + \frac{n+1}{n^2}\log \int \prod_{j=1}^n |x-x_j|^2 e^{-n\tilde{\phi}(x)} dT^*\nu(x)\right]\right) d\kappa_n. $$

\end{proposition}

\begin{proposition}[Pushing the real case on the sphere]
	Let $(z_1,\dots,z_n)$ be the zeros of $P_n$ in the real case, then the law of $(T(z_1),\dots,T(z_n))$ is:
\begin{multline*}
\sum_{k=0}^{\lfloor n/2 \rfloor} \frac{1}{Z_{n,k}} \frac{ \prod_{i<j}|x_i-x_j|\times \prod_{i=1}^{n} (1-|x_i|^2)}{ (\int \prod_{j=1}^n |x-x_j|^2 e^{-n\tilde{\phi}(x)} dT^*\nu(x))^{(n+1)/2}}  dL_{n,k}(x_1,\dots,x_n).
\end{multline*}
If we define $\rho_{n,k}$ as in Proposition \ref{pushreelkac}, this distribution can be written:
\begin{multline*}
	\sum_{k=0}^{\lfloor n/2 \rfloor} \frac{1}{Z_{n,k}}  \exp \left( -\beta_n\left[ \frac{1}{2}\mathcal{E}_{\neq}(\bar{\mu}_n)  -\frac{n+1}{2n^2} \log \int \prod_{j=1}^n |x-x_j|^2 e^{-n\tilde{\phi}(x)} dT^*\nu(x) \right]\right) d\rho_{n,k}.
	\end{multline*}
\end{proposition}

\begin{definition}
	We define $J_O: \mathcal{M}_1(\CC) \rightarrow \RR $ by:
	\begin{equation}
		J_O(\mu)= \sup_{z\in \CC} \{ \int \log |z-w|^2 d\mu(w) - \phi(z) \}
	\end{equation}
	and $J_{O,\mathcal{S}^2}: \mathcal{M}_1(\mathcal{S}^2) \rightarrow \RR$ by:
	\begin{equation}
		J_{O,\mathcal{S}^2}(\mu)= \sup_{x \in \mathcal{S}^2}  \{\int \log|x-y|^2 d\mu(y)-\tilde{\phi}(x)\}.
	\end{equation}
\end{definition}

\subsection{Steps 3 and 4}
\begin{itemize}
\item \emph{Rate function.} If we look at the proof of the large deviation principles for Kac polynomials and elliptic polynomials on the sphere, we see that the good definition of the rate function relies on the continuity of the function $J_{O,\mathcal{S}^2}$.  The proof of the continuity of this function is the same as in the previous cases under the assumptions that the support of $T^*\nu$ non-thin at all its points, which is one of our hypothesis. We replace the function $-U^{\mu}$ by the function $-U^{\mu} +\bar{\phi}$. Note that the set $A_{\varepsilon}$ would be replaced in general by the set:
\[	A_{\varepsilon}= \{ x \in \mathcal{S}^2 \mid -2 U^{\mu}(x) + \bar{\phi} \geq J_{O, \mathcal{S}^2}(\mu) - \varepsilon \}.\]

\item \emph{Upper Bound.} The proof of the large deviations upper bound relies on the Bernstein-Markov property \eqref{bersteinmarkov}, which is assumed to be true.

\item\emph{Lower Bound.} As $J_{O,\mathcal{S}^2}$ is continuous, we can reproduce exactly the proof of the lower bound for Kac polynomials.
\end{itemize}
In order to prove the large deviation principle for the normalized measures, we use the same technique as in \eqref{trick} and \eqref{limdetE}. We prove asymptotics for $Z_n$ and we deduce a uniform control over the constants $Z_{n,k}$. This control is necessary to mimic the work of Section \ref{getridofZ}.
Using the large deviations principle for non-normalized measures with the set $\mathcal{M}_1(\mathcal{S}^2)$, we have:
\begin{equation}
\varliminf_{n\to\infty} \frac{1}{\beta_n} \log Z_n = \varlimsup_{n\to\infty} \frac{1}{\beta_n} \log Z_n  = - \inf I_{O,\mathcal{S}^2}.
\end{equation}
Hence, doing exactly as in  we obtain the full large deviation principle for normalized measures in $\mathcal{M}_1(\mathcal{S}^2)$.
We also notice that, due to the definition of $Z_n$ given in Theorem \ref{generalroots}, we have:
\begin{equation} \label{limdetO}
\varliminf_{n\to\infty} \frac{1}{\beta_n} \log |A_n|^2 = \varlimsup_{n\to\infty} \frac{1}{\beta_n} \log |A_n|^2 = - \inf I_{O,\mathcal{S}^2}.
\end{equation}
Remembering the definition of $Z_{n,k}$ from Theorem \ref{generalroots} and the estimates given in Proposition \ref{limits}, we have:
\begin{equation}\label{asssumption2}
\lim_{n\to\infty} \frac{1}{\beta_n} \min Z_{n,k} = \lim_{n\to\infty} \frac{1}{\beta_n} \max  Z_{n,k} = -\frac{1}{2} \inf I_{O,\mathcal{S}^2}.
\end{equation}
This allows us to prove the large deviation principle for normalized measures as in Section \ref{getridofZ}. For general $\beta_n$ we need to assume a uniform control over the constants $Z_{n,k}$ given in \eqref{asssumption2}.

We end the proof of the large deviations principle exactly in the same way as we did in Section \ref{getridofZ}.

Once the large deviation principle proved in $\mathcal{M}_1(\mathcal{S}^2)$, we can prove Theorem \ref{LDPgeneral} and Theorem \ref{LDPgeneralreel}.

\begin{proof}[Proof of Theorem \ref{LDPgeneral}]
	Thanks to the inclusion principle \cite[Lemma 4.1.5]{dembozeitouni}, the random sequence $(T^*\mu_n)_{n\in \NN}$ satisfies a large deviation principle in $$\{ \mu \in \mathcal{M}_1(\mathcal{S}^2) \mid \mu(\{N\})=0\}$$ with speed $\beta_n$ and good rate function $I_{O,\mathcal{S}^2}$. Then by the contraction principle \cite[Theorem 4.2.1]{dembozeitouni} along $T^{-1}$, the sequence $(\mu_n)_{n \in \NN}$ satisfies a large deviation principle with the same speed and good rate function $I_O$ thanks to Proposition \ref{equivalence}. The contraction principle ensures that the function $I_O$ is a good rate function as $I_{O,\mathcal{S}^2}$ is a good rate function.
\end{proof}

\begin{proof}[Proof of Theorem \ref{LDPgeneralreel}]
	In the real case, the proof is exactly the same as Theorem \ref{LDPgeneral}. We use the inclusion principle and the contraction principle to obtain a large deviation principle with speed $\beta_n$ and good rate function $\tilde{I}_O$ (using again Proposition \ref{equivalence}).
\end{proof}

\section{Acknowledgements}
We would like to thank Ofer Zeitouni for his very helpful comments and remarks, as well as Djalil Chafaï for his valuable help in the construction of this article.
%\nocite{*}
\bibliographystyle{alpha} % D'autres styles sont disponibles. Notez que les distributions LaTeX n'incluent généralement pas de styles de bibliographies francisés ; vous aurez donc des bibliographies en anglais.
\bibliography{biblio} % Remplacer "biblio" par le nom de votre fichier de références bibliographiques.
\end{document}